\documentclass[10pt, reqno]{amsart}
\usepackage{amsfonts}
\usepackage{amsmath,amssymb}
\usepackage{mathrsfs}
\usepackage{hyperref}
\usepackage{graphicx}
\usepackage{amsthm}
\usepackage{color}
\usepackage{dsfont}
\usepackage{bbm, dsfont}
\usepackage{cite}
\numberwithin{equation}{section}

\allowdisplaybreaks

\newtheorem{theorem}{Theorem}[section]

\newtheorem*{theorem*}{Theorem}

\newtheorem{lemma}[theorem]{Lemma}

\theoremstyle{definition}
\newtheorem{definition}[theorem]{Definition}
\newtheorem{remark}[theorem]{Remark}

\theoremstyle{remark}

\begin{document}

\title[Global asymptotic behavior for gDNLS]{Global asymptotic behavior of solutions to the generalized derivative nonlinear Schr\"odinger equation}

\author[M. Shan]{Minjie Shan}
\email{smj@muc.edu.cn}
\address{College of Science, Minzu University of China, Beijing 100081, P. R. China}

\maketitle

\begin{abstract} This article is concerned with the global asymptotic behavior for the generalized derivative nonlinear Schr\"odinger (gDNLS) equation. When the nonlinear effect is not strong, we show  pointwise-in-time dispersive decay for solutions to the gDNLS  equation with small initial data in $H^{\frac{1}{2}+}(\mathbb{R})$ utilizing crucially  Lorentz-space improvements of the traditional Strichartz inequality. When the nonlinear effect is especially dominant, there exists a sequence of solitary waves  that  are arbitrary small in the energy space, which means the small data scattering is not true. However, there is evidence that it is not possible for the solitons to be localized in $L^{2}(\mathbb{R})$ and small in $H^{1}(\mathbb{R})$. With small and localized data assumption, we obtain global asymptotic behavior for solutions to the gDNLS  equation by using vector field methods combined with the testing by wave packets method.
\end{abstract}

\section{Introduction}              
We study the asymptotic behavior of solutions to the generalized derivative nonlinear Schr\"odinger (gDNLS) equation 
\begin{equation}
	\left\{
	\begin{aligned}
		&i\partial_{t}u +\partial_{x}^2 u +i\partial_{x}\big(|u|^{2\sigma}u\big) = 0, \\
		&u(0,x)=u_0(x),\ \ \ (t, x)\in \mathbb{R}\times\mathbb{R}, \label{gDNLS} \\
	\end{aligned}
	\right.
\end{equation}
where $u(t, x)$ is a complex-valued function and $\sigma>0$.  The gDNLS equation describes the propagation of Alfv\'en waves with small but finite amplitude in the magnetized plasma \cite{MOMT76, Mjol76}.

This equation enjoys the mass conservation law
	\begin{align}
M\big(u(t)\big):=\int |u(t)|^2dx=M\big(u_0\big)   \label{gDNLSmass}
	\end{align}
and energy conservation law
	\begin{align}
E\big(u(t)\big):=\int |u_x|^2-\frac{2\sigma+1}{\sigma+1}|u |^{2\sigma}\text{Im}(u_x\bar{u})+\frac{1}{\sigma+1}|u|^{4\sigma+2} dx=E\big(u_0\big).   \label{gDNLSenergy }
	\end{align}
Moreover, \eqref{gDNLS} admits the following scaling symmetry: If $u$  is a solution, then for any $\lambda>0$, so is
$$u_{\lambda}(t,x):=\lambda^{\frac{1}{2\sigma}}u(\lambda^2 t,\lambda x).$$
Note that
$$\|u\|_{\dot{H}^{s_c}(\mathbb{R})}=\|u_{\lambda}\|_{\dot{H}^{s_c}(\mathbb{R})}$$ 
deduces $s_c=\frac{1}{2}-\frac{1}{2\sigma}$.  We call $s_c=\frac{1}{2}-\frac{1}{2\sigma}$ the scaling-critical Sobolev index as the $\dot{H}^{s_c}$-norm is invariant under the scaling symmetry. 

When $\sigma=1$, \eqref{gDNLS} becomes 
\begin{align}
i\partial_{t}u +\partial_{x}^2 u +i\partial_{x}\big(|u|^{2}u\big) = 0. \label{DNLS}
	\end{align}
This is the standard derivative nonlinear Schr\"odinger equation (DNLS) which is mass-critical. The well-posedness and long-term behavior for the DNLS equation \eqref{DNLS} have been widely studied by many authors \cite{GT91, Hayashi93, Ozawa96, OT98, TF80, TF81}.  Hayashi and Ozawa \cite{HO92, HO94} obtained the local well-posedness for the DNLS equation \eqref{DNLS} in $H^1(\mathbb{R})$.  Using a gauge transformation and the Fourier restriction norm method,  Takaoka \cite{Takaoka99} showed local well-posedness in $H^s(\mathbb{R})$ for $s\geq \frac{1}{2}$. The local well-posedness result in $H^{\frac{1}{2}}(\mathbb{R})$ is sharp in the sense that the data map fails to be $C^3$ for $s < \frac{1}{2}$, see \cite{Takaoka01}. In \cite{BL01}, Biagioni and Linares showed that the data-to-solution mapping even fails to be locally uniformly continuous below $H^{\frac{1}{2}}(\mathbb{R})$ . Mosincat and Yoon \cite{MY20} considered unconditional well-posedness in $H^s$ for $s>\frac{1}{2}$. The local well-posedness in the Fourier-Lebesgue space was studied  by Gr\"unrock in \cite{Grunrock05}. For global well-posedness, based on two gauge transformations performed to remove the derivative in the nonlinear term,  Ozawa \cite{Ozawa96}  proved that \eqref{DNLS}  is globally well-posed in $H^1(\mathbb{R})$ assuming the smallness condition $\|u_0\|_{L^2}<\sqrt{2\pi}$. As the $L^2$ norm of solution to the DNLS equation \eqref{DNLS} is conserved, the smallness condition imposed on the initial data will force the energy to be positive via the sharp Gagliardo-Nirenberg inequality. Afterwards, the I-method \cite{Takaoka01, CKSTT01, CKSTT02, MWX11} was applied to refine the global well-posedness result under the same $L^2$-smallness condition. Long-term behavior and modified scattering were studied in \cite{GHLN13}. Taking use of the momentum conservation, Wu \cite{Wu13, Wu15} obtained the global well-posednss in $H^1(\mathbb{R})$  under a weaker condition $\|u_0\|_{L^2}<2\sqrt{\pi}$. This result shows a striking difference between DNLS and other mass critical equations like focusing generalized KdV and quintic focusing nonlinear Schr\"odinger equation. Under the same initial data condition, Guo and Wu \cite{GW17} proved that  \eqref{DNLS} is globally well-posed in $H^{\frac{1}{2}}(\mathbb{R})$. 

 Another interesting property of DNLS \eqref{DNLS} is that it is completely integrable. Using the inverse scattering method, Jenkins, Liu, Perry and Sulem \cite{JLPS18, JLPS20} proved global wellposedness for general initial conditions in the weighted Sobolev space $H^{2,2}(\mathbb{R})$. It is noteworthy that Bahouri and Perelman \cite{BP22} recently obtained global well-posedness in $H^{\frac{1}{2}}(\mathbb{R})$ without smallness assumption on the $L^2$-norm, relying on profile decomposition, and also crucially on complete integrability.  Low regularity conservation laws  for DNLS in Besov spaces with the full subcritical regularity index  was obtained by Klaus and Schippa \cite{KS22} using the perturbation determinant method introduced by Killip, Visan and Zhang \cite{KVZ18}. Tang and Xu \cite{TX20} showed the corresponding microscopic conservation laws for the Schwartz solutions of DNLS with small mass.  Killip, Ntekoume and Visan \cite{KNV23} obtained global well-posedness in $H^{s}(\mathbb{R})$ for $\frac{1}{6}\leq s< \frac{1}{2}$ with $\|u_0\|_{L^2}<2\sqrt{\pi}$. Taking advantage of the commuting flows method, local smoothing effect and tightness estimates, Harrop-Griffiths, Killip, Ntekoume and Visan \cite{HGKNV22}  further proved that  DNLS  \eqref{DNLS}  is globally well-posed in  $L^{2}(\mathbb{R})$, and more generally in  $H^{s}(\mathbb{R})$ for $0\leq s\leq \frac{1}{2}$.  The ill-posedness in $H^{s}(\mathbb{R})$ for $s<0$ is a direct reslut from the self similar solutions constructed in \cite{FGO20}. 
 
Now we turn to  well-posedness for the  Cauchy problem of the gDNLS equation \eqref{gDNLS}. Hao \cite{Hao07} proved local well-posedness in $H^{\frac{1}{2}}(\mathbb{R})$ for $\sigma \geq \frac{5}{2}$ by using  gauge transformation and  Littlewood-Paley decomposition.  Ambrose and Simpson \cite{AS15} showed local well-posedness in $H^{1}(\mathbb{R})$ for $\sigma \geq 1$. This local well-posedness result was improved to $H^{\frac{1}{2}}(\mathbb{R})$ by Santos \cite{Santos15}.
Hayashi and Ozawa \cite{HO16} obtained local well-posedness in  $H^{2}(\mathbb{R})$ for $\sigma\geq \frac{1}{2}$ and global well-posedness in  $H^{1}(\mathbb{R})$ for $\sigma\geq1$. For $0<\sigma<\frac{1}{2}$, Linares, Ponce and Santos \cite{LPS19a, LPS19b} showed the local well-posedness for a class of data of arbitrary size in some weighted Sobolev spaces. There are only few results with regard to global well-posedness. Hayashi and Ozawa \cite{HO16} showed the global existence without uniqueness in  $H^{1}(\mathbb{R})$ for $0<\sigma<1$.  Fukaya, Hayashi and Inui \cite{FHI17} obtained a sufficient condition of initial data for global well-posedness  in  $H^{1}(\mathbb{R})$ for $\sigma>1$.  When $\sigma\geq2$, Bai, Wu and Xue \cite{BWX20} showed that \eqref{gDNLS} is globally well-posed and scatters in $H^{s}(\mathbb{R})$ ( $\frac{1}{2} \leq s \leq 1$)  with small initial data.  Moreover, they proved that there exists a sequence of solitary waves which are arbitrary small in $H^{1}(\mathbb{R})$ when $0<\sigma<2$, which is against the small data scattering statement. Their result suggested the exponent $\sigma \geq2$ is optimal for small data scattering. See Lemma  \ref{BWX20Result} for a more detailed description. The existence of wave operator for \eqref{gDNLS} with $\sigma\in \mathbb{R}$ and $\sigma\geq 3$ was proved by Bai and Shen \cite{BS23}. Regarding  the semilinear Schr\"odinger equation, there are two important exponents named short range exponent and the Strauss exponent. The short range exponent is closely related to global well-posedness, and the Strauss exponent is usually used to determine whether scattering occurs. Specifically, if the nonlinear power is larger than the short range exponent $3$ ($1 +\frac{2}{n}$ for general dimensions $n$), one has small data global well-posedness and the existence of wave operator (see for examples \cite{CW92, GOV94, Nakanishi01}); if the nonlinear power is larger than the Strauss exponent $\frac{\sqrt{17}+3}{2}\approx 3.56$ ($\frac{\sqrt{n^2+12n+4}+n+2}{2}$ for general dimensions $n$), one has small data scattering (see for examples \cite{Strauss81}). However, as far as we know there is no such general result for non-semilinear Schr\"odinger equation. Comparing with the semilinear Schr\"odinger equations, from the result of Bai, Wu and Xue \cite{BWX20}, we see that the optimal scattering exponent for the nonlinear Schr\"odinger equation with derivatives is larger than the short range exponent and the Strauss exponent. This once again demonstrates the huge difference between the derivative nonlinear Schr\"odinger equation  and the nonlinear Schr\"odinger equation.

For the gDNLS equation with more general nonlinear term 
\begin{align}
i\partial_{t}u +\Delta u  = P(u, \bar{u}, \partial_x u, \partial_x \bar{u}), \label{gGDNLS}
	\end{align}
Kenig, Ponce and Vega \cite{KPV93} showed that  \eqref{gGDNLS} is locally well-posed in $H^{\frac{7}{2}}(\mathbb{R})$ with small initial data, when  $P$ is a polynomial of the form $P(z)=\sum_{d\leq |\alpha| \leq l}C_{\alpha}z^{\alpha}$ and $l, d$ are integers with $l\geq d \geq 3$. Gr\"unrock \cite{Grunrock00} proved local  well-posedness for \eqref{gGDNLS} in $H^{s}(\mathbb{R})$ with $s>\frac{1}{2}-\frac{1}{d-1}$ when $P=\partial_x(\bar{u}^d)$ and $s>\frac{3}{2}-\frac{1}{d-1}$ when $P=(\partial_x\bar{u})^d$ respectively. Hirayama \cite{ Hirayama15} showed the small data global well-posedness  for \eqref{gGDNLS}  in $H^{s}(\mathbb{R})$ with $s>\frac{1}{2}-\frac{1}{d-1}$ when $P=\partial_x(\bar{u}^d)$. By using a decomposition technique instead of using a gauge transformation to deal with the loss on derivative in nonlinearity,  Pornnopparath \cite{Pornn18} systematically studied the local and global solutions of \eqref{gGDNLS} in low-regularity Sobolev spaces. For higher dimension and more related theories, see  \cite{Bej06, Bej08, BT08,  Wang11} and the references therein.

Once the global well-posedness was obtained, the next natural questions to ask are: How is the global asymptotic behavior in time of the solution?  Whether the solution to the nonlinear equation exhibits the same dispersive decay as the solution to the corresponding linear equation?

For the nonlinear evolution equation \eqref{gDNLS}, scattering  means that there exist unique $u_0^{\pm}\in H^s(\mathbb{R})$ for some $s\in\mathbb{R}$ such that
	\begin{align}
		\lim_{t\to \pm \infty}\left\|u(t)- e^{it\partial^2_x}u_0^{\pm}\right\|_{H^{s}_{x}}=0, \nonumber
	\end{align}
we intuitively expect that the long-time asymptotic development of solution $u(t)$ to  \eqref{gDNLS} behaves like a solution to the linear Schr\"odinger equation which has the dispersive decay
	\begin{align}
		  \| e^{it\partial^2_x}u_0 \|_{L^{p}_{x}}\lesssim_p |t|^{-(\frac{1}{2}-\frac{1}{p})} \|u_0 \|_{L^{p'}_{x}}, \label{SchStri}
	\end{align}
where $2\leq p \leq \infty$ and $\frac{1}{p}+\frac{1}{p'}=1$. \eqref{SchStri} is derived by interpolating between the mass identity  and the dispersive estimate 
\begin{align}
		  \| e^{it\partial^2_x}u_0 \|_{L^{\infty}_{x}}\lesssim  |t|^{-\frac{1}{2}} \|u_0 \|_{L^{1}_{x}}. \label{SchStri11}
	\end{align}

Dispersive decay property for nonlinear dispersive equations has been widely studied in the past forty years.  Utilizing the Morawetz estimate, Lin and Strauss \cite{LinStr78}  obtained the decay estimate of $L^{\infty}$ norm which was further used to show the existence of the scattering operator for the 3D nonlinear Schr\"odinger equation (NLS). It is also possible to apply vector filed methods and commutator type estimates to get decay estimates, see for example, \cite{Klai85}. But it may require the initial datum lives in some weighted $L^2$ space. In early stage,  quantitative decay estimates for solutions to nonlinear evolution equations were achieved under strong regularity and decay hypotheses; see, for example,  \cite{Klai85, KlaiPon83,  MorStr72, Shatah82, Seg83, HayTsu86} as well as the references therein.

Further progress has considerably weakened the regularity and/or decay conditions needed to construct local solutions. For the energy-critical NLS, Fan and Zhao \cite{FZ21},  Guo, Huang and Song \cite{GHS23} respectively obtained $L^{\infty}_x$ dispersive decay when $u_0\in H^3(\mathbb{R}^3)$. For the cubic NLS with initial data $u_0\in H^1(\mathbb{R}^3)$, Fan, Staffilani and Zhao \cite{FSZ24} got point-wise decay estimates. What's more, they showed that the  randomization of the initial data can be used to replace the $L^1$-data assumption. See \cite{FZ23} for the necessity of the $L^1$-data assumption. For the mass-critical NLS  with initial data $u_0\in L^2(\mathbb{R}^d)\cap  L^{r'}(\mathbb{R}^d)$ ($d=1,2,3$), Fan, Killip, Visan  and Zhao \cite{FKVZ24} proved dispersive decay by using Lorentz-space improvements of the traditional Strichartz inequality (for $d=1,2$) and a subtle decomposition of the nonlinearity (for $d=3$). Their work was optimal in the sense that no auxiliary assumptions were made besides finiteness of the critical norm.  For the energy-critical NLS,  Kowalski \cite{Kowal24} recently showed dispersive decay with initial data $u_0$ in the scaling-critical Sobolev space $ \dot{H}^1$ when $d=3,4$ .  For work on dispersive estimates for various wave equations, we direct the  reader to \cite{Klai80, Klai85, KlaiPon83, Kowal25, Loo21, Loo22,  MorStr72,  Seg83}. In \cite{Shan24}, we showed the dispersive decay for solutions to the mass-critical gKdV equation with  $u_0\in H^{\frac{1}{4}}\cap L^{1}$ . The fact that the nonlinear term in the  KdV model  contains derivatives   distinguishes our results from prior works in this direction. We need smooth estimates to handle nonlinear terms with derivatives.

In this paper, we prove analogues of \eqref{SchStri11} for the  generalized derivative nonlinear Schr\"odinger equation. We are now ready to state the main results.

When the  nonlinear effects are weak, dispersive decay is proved for solutions to \eqref{gDNLS} with initial data in $H^{1-}$. 
\begin{theorem} \label{MainResult1}
Assume that $\sigma> 2$ and $u(t)$ is the global  solution to the  gDNLS  equation  \eqref{gDNLS} with initial data $u_0\in H^{\frac{1}{2}+\frac{\sigma-2}{2\sigma-3}} \cap L^{1}$ and $\|u_0\|_{H^{\frac{1}{2}+\frac{\sigma-2}{2\sigma-3}}}\ll1$, then there exists a constant $C=C\big(\|u_0\|_{ H^{\frac{1}{2}+\frac{\sigma-2}{2\sigma-3}}}\big)$, such that 
	\begin{align}
\|u(t,x)\|_{L^{\infty}_{x}}\leq C |t|^{-\frac{1}{2}} \|u_0\|_{L^{1}_{x}}.   \label{MainResult1a}
	\end{align}
If $\sigma= 2$ and $u(t)$ is the global  solution to the  gDNLS  equation  \eqref{gDNLS} with initial data $u_0\in H^{\frac{1}{2}+} \cap L^{1}$ and $\|u_0\|_{H^{\frac{1}{2}+}}\ll1$, then \eqref{MainResult1a} holds true.
\end{theorem}

However, when the  nonlinear effects are strong, from \cite{BWX20} (see Lemma \ref{BWX20Result}), we know that there exist a class of solitary wave solutions for \eqref{gDNLS}. Therefore, the nonlinear solution is impossible to decay like the linear one globally in time. In fact, the \emph{soliton resolution conjecture} asserts, roughly speaking, that any reasonable solution to a dispersive  equation eventually resolves into a superposition of a dispersive component plus a number of ``solitons". 

If solitons exist, one can distinguish two stages to consider dispersive decay bounds for a nonlinear equation. The first stage is to add some conditions on the initial data so that the solitons case is excluded. Then, the nonlinear solution may decay like the linear
solution globally in time. The second stage is to study the time scale up to which the solution will satisfy linear dispersive decay bounds.  In other words, we expect to get the optimal quartic time scale that marks the earliest possible emergence of solitons.  Some work on this topic can be found in \cite{IfKoTa19, IfSa23, IfTa19} for KdV, intermediate long wave, and Benjamin-Ono equations. 

In this article we aim to describe the first of the two stages above.

\begin{theorem} \label{MainResult2}(a)(Dispersive decay) Let $0<\epsilon \ll 1$. Assume that $1\leq\sigma<2$ and $u(t)$ is the global  solution to \eqref{gDNLS} with initial data $u_0$ satisfying 
	\begin{align}
 \|xu_0\|_{H^1}+ \|u_0\|_{L^2}\leq \epsilon . \nonumber
	\end{align}
Then we have the following bound on $Lu$
	\begin{align}
\|Lu\|_{H^{1}_{x}}\lesssim  \epsilon \langle t \rangle^{C \epsilon^{\frac{3}{2}}} , \label{MainResult2a}
	\end{align}
and the dispersive bounds
	\begin{align}
\|u\|_{L^{\infty}_{x}} \lesssim  \epsilon \langle t \rangle^{-\frac{1}{2}} \hspace{10mm} \text{and} \hspace{10mm}  \|u_x\|_{L^{\infty}_{x}}\lesssim  \sqrt{\epsilon}\langle t \rangle^{-\frac{1}{2}} \label{MainResult2b}
	\end{align}
for all $t\in \mathbb{R}$, where $L:=x+2it\partial_x$ and $\langle t \rangle=(1+t^2)^{-\frac{1}{2}}$.

(b)(Asymptotic behavior)  Assume that $u(t)$ is the global  solution to the DNLS ( with $\sigma=1$ in \eqref{gDNLS}) as in part (a). Then there exists a function $W\in H^{1-C\epsilon^{\frac{3}{2}}}(\mathbb{R})$ such that
	\begin{align}
u(t,x)&=\frac{1}{\sqrt{t}}e^{i\frac{x^2}{4t}}W(\frac{x}{t})e^{-i\frac{x^2}{2t}\log t|W(\frac{x}{t})|^2}+err_x, \label{MainResult2c} \\
\widehat{u}(t,\xi)&= e^{-it\xi^2}W(2\xi)e^{-i\xi\log t|W(2\xi)|^2}+err_{\xi}, \label{MainResult2d}
	\end{align}
where
	\begin{align}
err_x & \in\epsilon\left(O_{L^{\infty}}(\langle t \rangle^{-\frac{3}{4}+C\epsilon^{\frac{3}{2}}})\cap O_{L^{2}_x}(\langle t \rangle^{-1+C\epsilon^{\frac{3}{2}}})\right), \notag  \\
err_{\xi}& \in\epsilon\left(O_{L^{\infty}}(\langle t \rangle^{-\frac{1}{4}+C\epsilon^{\frac{3}{2}}})\cap O_{L^{2}_{\xi}}(\langle t \rangle^{-\frac{1}{2}+C\epsilon^{\frac{3}{2}}})\right). \nonumber 
	\end{align}
If $u(t)$ is the global  solution to  the gDNLS \eqref{gDNLS} with $1<\sigma<2$ as in part (a). Then there exists a function $\widetilde{W}\in H^{1}(\mathbb{R})$ such that
	\begin{align}
u(t,x)&=\frac{1}{\sqrt{t}}e^{i\frac{x^2}{4t}}\widetilde{W}(\frac{x}{t})e^{-i\frac{x^2}{2} t^{-\sigma}|\widetilde{W}(\frac{x}{t})|^{2\sigma}}+err_x, \label{MainResult2e} \\
\widehat{u}(t,\xi)&= e^{-it\xi^2}\widetilde{W}(2\xi)e^{-i\xi t^{1-\sigma}|\widetilde{W}(2\xi)|^{2\sigma}}+err_{\xi}, \label{MainResult2f}
	\end{align}
where
	\begin{align}
err_x & \in\epsilon\left(O_{L^{\infty}}(\langle t \rangle^{-\frac{3}{4}})\cap O_{L^{2}_x}(\langle t \rangle^{-1})\right), \notag  \\
err_{\xi}& \in\epsilon\left(O_{L^{\infty}}(\langle t \rangle^{-\frac{1}{4}})\cap O_{L^{2}_{\xi}}(\langle t \rangle^{-\frac{1}{2}})\right). \nonumber 
	\end{align}
\end{theorem}

	\begin{remark}
When $1\leq\sigma<2$, 
$$u_{\omega, c}(t)=e^{i\omega t}\phi_{\omega, c}(x-ct)$$
 are the solitons to \eqref{gDNLS}, where the parameters $c^2<4\omega$, and 
$$\phi_{\omega, c}(x)=\varphi_{\omega, c}(x)\exp\left\{\frac{c}{2}ix-\frac{i}{2\sigma+2}\int_{-\infty}^x\phi^{2\sigma}_{\omega, c}(y)dy\right\}$$
with
$$\varphi_{\omega, c}(x)=\left\{\frac{(\sigma+1)(4\omega-c^2)}{2\sqrt{\omega}\cosh(\sigma\sqrt{4\omega-c^2}x)-c}\right\}^{\frac{1}{2\sigma}}.$$
From \cite{BWX20}, we have 
\begin{align}
& \hspace{15mm}  \|\partial_x \phi_{\omega, c}\|^2_{L_x^2}=\omega\|\phi_{\omega, c}\|^2_{L_x^2},  \hspace{9mm} \|x \phi_{\omega, c}\|^2_{L_x^2}\|\partial_x \phi_{\omega, c}\|^2_{L_x^2} \geq 1,\notag \\
&\|\phi_{\omega, c}(t)\|^2_{L_x^2}=\|\varphi_{\omega, c}\|^2_{L_x^2}=\frac{2}{\sigma}\left(\frac{\sigma+1}{2\sqrt{\omega}}\right)^{\frac{1}{\sigma}}(4\omega-c^2)^{\frac{1}{\sigma}-\frac{1}{2}}\int_0^{\infty} \left(\frac{1}{\cosh x-\frac{c}{2\sqrt{\omega}}}\right)^{\frac{1}{\sigma}}dx, \nonumber
	\end{align}
 where $\|\phi_{\omega, c}(t)\|_{L_x^2}$ can be small by taking $c\to -2\sqrt{\omega}$. This shows that it is not possible for the solitons to be localized in $L^2$ and small in $H^1$.
	\end{remark}

\begin{remark} An interesting question is what happens for $\frac{1}{2}\leq\sigma<1$.   $\sigma=1$ is the optimal case
that our method can handle. Intuitively speaking, the optimal decay of the potential is
$$\||u|^{2\sigma}\|_{L_x^{\infty}}\sim \langle t \rangle^{-\sigma}.$$ 
Then, the energy method tells us the growth rate of  $\|Lu\|_{L_x^{2}}$ and $\|Lu_x\|_{L_x^{2}}$ is $e^{\int_0^t \langle s \rangle^{-\sigma}ds}$. It is easy to see that $\|Lu\|_{L_x^{2}}$ and $\|Lu_x\|_{L_x^{2}}$ are  bounded uniformly in time if $\sigma>1$,  and these two quantities increase at a polynomial rate if $\sigma=1$. Hence, for $\sigma\geq 1$, the difference between the solution $u$ to  \eqref{gDNLS} and the asymptotic profile $\gamma$ is controllable, see \eqref{DifferenceBounds2a}-\eqref{DifferenceBounds3b}. But, if $\sigma<1$, the energy bounds for $Lu$ and $ Lu_x $ grow  at an exponential rate. Our method will be of invalidity.
	\end{remark}

\begin{remark}
\eqref{MainResult2b} and \eqref{MainResult2e} had been obtained by Hayashi and Naumkin \cite{HN97} for the DNLS with initial data $u_0\in H^{\frac{3}{2}+, 0} \cap H^{1,\frac{1}{2}+}$ and $\|u_0\|_{H^{\frac{3}{2}+, 0}}+\|u_0\|_{H^{1,\frac{1}{2}+}}\leq \epsilon$, where $H^{m, s}:=\{f\in \mathcal{S'}; \|f\|_{H^{m, s}}=\|(1+x^2)^{s/2}(1-\partial_x^2)^{m/2}f|_{L^{2}}<\infty\}$. From the identity
$$xe^{it\partial_x^2}u_0=e^{it\partial_x^2}(xu_0)-2ite^{it\partial_x^2}(\partial_x u_0),$$
we see that the regularity of solutions to the linear Schr\"odinger equation is equivalent to the decay rate of the solutions.
Hence, Theorem \ref{MainResult2} can be regarded as an improvement of the result gave in \cite{HN97}.
	\end{remark}

\begin{remark}
Very recently, Byars \cite{Bya24} got the dispersive decay estimate \eqref{MainResult2b} for the DNLS under the assumption  $\|xu_0\|_{H^{1}}+\|u_0\|_{H^{5}}\leq \epsilon$. Compared to her result, our result presented in Theorem \ref{MainResult2}  greatly lowers the regularity.
	\end{remark}

Let us turn now to a brief overview of our methods. For the case $\sigma>2$, we use global a prior bounds showed in Lemma \ref{BWX20Result} and the global Lorentz-Strichartz estimates to get the dispersive decay. What we focus on is applying local smoothing effects to eliminate as many derivatives of the nonlinear term  as possible. For the case $1\leq\sigma<2$, we wold like to obtain the dispersive decay from Klainerman-Sobolev type inequalities provided that energy estimates for $u$ and $Lu$ are bounded uniformly in time. This is the main idea of the vector field method. However, in the setting we consider $\|Lu\|_{L^2}$ slowly increases over time, which prevents such a direct argument. This difficulty can be overcome by making use of the testing by wave packets method introduced by Ifrim and Tataru \cite{IfTa24}. Via constructing an asymptotic equation, we acquire more information for the global dynamics of the solution. On the one hand, the asymptotic equation is an ODE of which the solution can be easily obtained. On the other hand, the wave packets method better balances the scales of Fourier space localization and physical space localization, in such a way that linear and nonlinear matching errors become comparable. The asymptotic equation we find is indeed a good approximation of the original nonlinear equation.

\textbf{Notation.} Given $A, B \geq 0$,  $A\lesssim B$ means that $A\leq C \cdot B$ for an absolute constant $C>0$.  $A\gg B$ means that  $A>C \cdot B$ for a very large positive constant $C$.  We write $c+\equiv c+\epsilon$ and $ c-\equiv c-\epsilon$ for some  $0<\epsilon\ll 1$. Given a function $u$, we denote $\mathscr{F} u $ or $\widehat{u}$ its Fourier transform and denote $\mathscr{F}^{-1} u $  its Fourier inverse transform.  For $1\leq p,q\leq\infty$, define 
$$\|f\|_{L^p_tL^q_{x}}=\left(\int_{\mathbb{R}}\Big(\int_{\mathbb{R}}|f(t,x)|^qdx \ \Big)^{p/q}dt\right)^{1/p}$$
with the usual modifications if either $p=\infty$ or $q=\infty$.  Similar definitions and considerations may be made interchanging the variables $x$ and $t$. Define $D^sf$ and $J^sf$ as
	\begin{align}
D^s_x f =\mathscr{F}^{-1}|\xi|^{s}\hat{f}(\xi), \hspace{10mm} J^s_x f =\mathscr{F}^{-1}(1+\xi^2)^{s/2}\hat{f}(\xi) 	\nonumber
\end{align} 
for $s>0$.

\textbf{Organization of the paper.} In Section 2 we state some preliminary results. Then, in Section 3  we prove Theorem  \ref{MainResult1} by using dispersive estimate, local smoothing effect and  the Lorentz space improvement estimate. Section 4 is devoted to the proof of Theorem  \ref{MainResult2}. We first recall the vector field method and wave packet in Subsection 4.1 where the asymptotic profile is bounded. And, we find an approximate ODE dynamics for the asymptotic profile. Then, in Subsection 4.2, we get the energy bounds for $Lu$ and $Lu_x$. The difference between the solution to the gDNLS and the asymptotic profile can be controlled by these  energy bounds. Finally, in Subsection 4.3, we  use the bootstrap argument to imply the desired $L^{\infty}$-norm estimate. Subsequently, the asymptotic behavior of the solution can be derived from solving the approximate ODE and estimating error terms.

\subsection*{Acknowledgements} M.S is partially supported by the National Natural Science Foundation of China, grant numbers: 12101629 and 12371123.

\section{Preliminary}                                


We recall the global well-posedness and scattering results for \eqref{gDNLS} and some useful estimates, such as Lorentz-Strichartz estimates, local smoothing estimates, maximal function estimates and interpolation inequality.

\begin{lemma}[see \cite{BWX20}]  \label{BWX20Result}
Assume that $\sigma\geq 2$, $\frac{1}{2}\leq s \leq 1$ and $u_0\in H^s(\mathbb{R})$. Then there exists a constant $\delta_0>0$, such that if  $\|u_0\|_{H^s(\mathbb{R})}\leq \delta_0 $, then the corresponding solution  $u$ to \eqref{gDNLS} is global in time and satisfies
	\begin{align}
\|u \|_{X}\lesssim \|u_0\|_{H^s(\mathbb{R})}  \label{BWX20Result1}
	\end{align}
where
\begin{align}
\|u \|_{X}:=&\|u\|_{L_t^{\infty}H_x^s }+\|\partial_x u\|_{L_x^{\infty}L_t^2} +\sup\limits_{q\in [4, N_0]}\|  u\|_{L_x^{q}L_t^{\infty}}+\|u\|_{L_t^{4}L_x^{\infty} }\notag \\
&+\big\|D_x^{s-\frac{1}{2}}\partial_x u\big\|_{L_x^{\infty}L_t^{2} } +\big\|D_x^{s-\frac{1}{2}} u\big\|_{L_x^{4}L_t^{\infty} } +\big\|D_x^{s-\frac{1}{2}} u\big\|_{L_t^{4}L_x^{\infty} }.
  \label{BWX20Result2}
	\end{align}
Here $N_0$ is any fixed arbitrary large parameter. Moreover, there exists a unique $u_{\pm}$ such that for any $0\leq s'<s$,
\begin{align}
\|u(t)- e^{it\partial^2_x}u_{\pm}\|_{H^{s'}(\mathbb{R})}\to 0 \hspace{10mm}  \text{as} \hspace{10mm}   t \to \infty.\nonumber
	\end{align}

However, if  $0<\sigma< 2$, then there exist a class of solitary wave solutions $\{\phi_c\}$ satisfying
$\|\phi_c\|_{H^1(\mathbb{R})}\to 0$
when $c$ tends to some endpoint, which is against the small data scattering statement.
\end{lemma}

\begin{lemma}[see (4.21) and (4.24)  in \cite{BWX20}]  \label{BWX20Result22222}
Let $\sigma\geq 2$ and $\frac{1}{2}<s\leq 1$. Assume that $u$ is the global solution given in Lemma \ref{BWX20Result}, then
	\begin{align}
\left\|D^{s-\frac{1}{2}}_x\left(|u|^{2\sigma}\partial_x u\right) \right\|_{L_x^1L_t^2}\lesssim \|u\|^{2\sigma+1}_{X} , \label{BWX20Result22222a}\\
\left\|D^{s-\frac{1}{2}}_x\left(|u|^{2\sigma}\partial_x u\right) \right\|_{L_t^1L_x^2}\lesssim \|u\|^{2\sigma+1}_{X} , \label{BWX20Result22222b}
	\end{align}
 where $\|u\|_{X}$ is defined in \eqref{BWX20Result2}. 
\end{lemma}

Let us recall the well-known dispersive estimates and Strichartz estimates  for  the Schr\"odinger semi-group.

\begin{lemma}[Dispersive estimates] 
	Let $2\leq r\leq \infty$ and  $\frac{1}{r}+\frac{1}{r'}=1$. We have
\begin{align}
\big\|e^{it\partial^2_x} u_0 \big\|_{ L^{r}_{x}}\lesssim t^{-(\frac{1}{2}-\frac{1}{r})}\|u_0 \|_{L^{r'}_{x}}. \label{dispEsti} 
\end{align}
\end{lemma}

We say $(q,p)$ is a Schr\"odinger $d$-admissible pair if
$$4 \leq q \leq \infty, \hspace{5mm}2 \leq  p \leq \infty  \hspace{5mm} \text{and}  \hspace{5mm} \frac{2}{q}+\frac{1}{p} = \frac{1}{2}.$$

\begin{lemma}[Strichartz estimates, see \cite{Takaoka01} ] 
Let  $(q,p)$ and  $(\tilde{q},\tilde{p})$   be Schr\"odinger admissible pairs. Then
\begin{align}
  \big\|e^{it\partial^2_x} u_0 \big\|_{L^q_{t}L^p_{x}}&\lesssim \|u_0 \|_{L^{2}_{x}},  \label{StriEsti1} \\
 \left\|\int_0^t e^{i(t-s)\partial^2_x} F(x, s)ds \right\|_{L^q_{t}L^p_{x}}&\lesssim \|F \|_{L^{\tilde{q}'}_{t}L^{\tilde{p}'}_{x}}   \label{StriEsti2}
	\end{align}
where $\frac{1}{\tilde{q}}+\frac{1}{\tilde{q}'} =\frac{1}{\tilde{p}}+\frac{1}{\tilde{p}'} =1$.
\end{lemma}

\begin{definition}[Lorentz spaces] Let $1\leq p<\infty$ and $1\leq q\leq\infty$. The Lorentz space $L^{p,q}(\mathbb{R})$ is the space of measurable functions $f:\mathbb{R} \to \mathbb{C}$ for which the quasi-norm
$$\|f\|_{L^{p,q}(\mathbb{R})}=p^{\frac{1}{q}}\left\| \lambda \big| \{x\in \mathbb{R}: |f(x)|>\lambda \}\big|^{\frac{1}{p}} \right\|_{L^{q}((0, \infty), \frac{d\lambda}{\lambda})}$$
is finite. Here, $|A|$ denotes the Lebesgue measure of the set $A\subseteq \mathbb{R}$.
\end{definition}

Note that $L^{p,p}(\mathbb{R})\simeq L^{p}(\mathbb{R})$, and $L^{p,\infty}(\mathbb{R})$ coincides with the weak $L^{p}(\mathbb{R})$ space. If $1\leq p<\infty$ and $1\leq q<r\leq\infty$, then  $L^{p,q}(\mathbb{R})\hookrightarrow  L^{p,r}(\mathbb{R})$. 

\begin{lemma}[H\"older inequality in Lorentz spaces] Let $1\leq p,p_1,p_2<\infty$, $1\leq q,q_1,q_2\leq\infty$, and $\frac{1}{p}=\frac{1}{p_1}+\frac{1}{p_2}$, $\frac{1}{q}=\frac{1}{q_1}+\frac{1}{q_2}$. Then,
\begin{align}\|fg\|_{L^{p,q}(\mathbb{R})}\lesssim \|f\|_{L^{p_1,q_1}(\mathbb{R})}\|g\|_{L^{p_2,q_2}(\mathbb{R})}. \label{Holder- Lor1}
	\end{align}
\end{lemma}

Applying Hardy-Littlewood-Sobolev inequality in Lorentz spaces yields the following Lorentz-space improvement.
\begin{lemma}[Lorentz-Strichartz estimates]\label{Lor-Stri}
	Let $(q,p)$ and  $(\tilde{q},\tilde{p})$  be Schr\"odinger admissible pairs. Then
\begin{align}
 \big\|e^{it\partial^2_x} u_0 \big\|_{L^{q,2}_{t}L^p_{x}}&\lesssim \|u_0 \|_{L^{2}_{x}},  \label{Lor-Stri1} \\
 \left\|\int_0^t e^{i(t-s)\partial^2_x} F(x, s)ds \right\|_{L^{q,2}_{t}L^p_{x}}&\lesssim \|F \|_{L^{\tilde{q}',2}_{t}L^{\tilde{p}'}_{x}}   \label{Lor-Stri2}
	\end{align}
where $\frac{1}{\tilde{q}}+\frac{1}{\tilde{q}'} =\frac{1}{\tilde{p}}+\frac{1}{\tilde{p}'} =1$.
\end{lemma}

By using the global space-time bounds in  Lemma \ref{BWX20Result} and Lorentz-Strichartz estimates in  Lemma \ref{Lor-Stri}, we obtain the following  lemma  record space-time bound  in Lorentz spaces  for the solution  to \eqref{gDNLS}. The  Lorentz improvement  plays a key role in the proof  of Theorem \ref{MainResult1}.

\begin{lemma}\label{LorImpr}
Let $\sigma\geq 2$, $0 <\varepsilon \ll1$ and $(q, p)=\left(\frac{4\sigma(2+\varepsilon)-12}{\varepsilon}, \frac{2\sigma(2+\varepsilon)-6}{\sigma(2+\varepsilon)-3-\varepsilon}\right)$ be a Schr\"odinger admissible pair. Assume that $u(t)$ is the global solution to \eqref{gDNLS} given in Lemma  \ref {BWX20Result}  with initial data $u_0\in H^{\frac{1}{2}+\frac{\sigma-2}{2\sigma-3}}(\mathbb{R})$, then we have the global space-time bound
	\begin{equation}
 \| u |_{L^{q,2}_{t}L^{p}_{x}}+\big\|D_x^{\frac{\sigma-2}{2\sigma-3}}u\big \|_{L^{q,2}_{t}L^{p}_{x}}\leq C(\|u_0\|_{ H^{\frac{1}{2}+\frac{\sigma-2}{2\sigma-3}}}) . \label{LorImpr1}
	\end{equation}
\end{lemma}

	\begin{proof}
 Using the Duhamel formula 
$$u(t)=e^{it\partial^2_x}u_0+\int_{0}^{t} e^{i(t-s)\partial^2_x}\partial_{x}\big(|u|^{2\sigma}u\big)(s) ds$$
and Lemma \ref{Lor-Stri}, we may estimate
	\begin{align}
\| u\|_{L^{q,2}_{t}L^{p}_{x}} &\lesssim \big\|e^{it\partial^2_x}u_0\big \|_{L^{q,2}_{t}L^{p}_{x}}+\left\|\int_{0}^{t} e^{i(t-s)\partial^2_x}\partial_{x}\big(|u|^{2\sigma}u\big)(s) ds\right \|_{L^{q,2}_{t}L^{p}_{x}} \notag \\ 
&\lesssim  \|u_0\|_{L^{2}_x}+\left\|\partial_{x}\big(|u|^{2\sigma}u\big)\right \|_{L^{1,2}_{t}L^{2}_{x}} \notag \\ 
&\lesssim  \|u_0\|_{L^{2}_x}+\left\||u|^{2}\partial_x u\right \|_{L^{2}_{t}L^{2}_{x}} \big\||u|^{2(\sigma-1)}\big\|_{L^{2}_{t}L^{\infty}_{x}} \notag \\ 
&\lesssim  \|u_0\|_{L^{2}_x}+ \|u \|^2_{L^{4}_{x}L^{\infty}_{t}} \|\partial_x u \|_{L^{2}_{x}L^{\infty}_{t}}  \| u  \|^{2(\sigma-1)}_{L^{4(\sigma-1)}_{t}L^{\infty}_{x}} \notag  
	\end{align}
which by global well-posedness result in Lemma \ref{BWX20Result} further yields 
	\begin{align}
\| u\|_{L^{q,2}_{t}L^{p}_{x}}\lesssim  \|u_0\|_{L^{2}_x}+\|u\|^{2\sigma+1}_{X}<\infty.  \label{LorImpr2}
	\end{align}

It follows from Lorentz-Strichartz estimates in  Lemma \ref{Lor-Stri}, the embedding inequality and \eqref{BWX20Result22222b}  with $s=\frac{1}{2}+\frac{\sigma-2}{2\sigma-3}<1$ that
	\begin{align}
&\big\|D_x^{\frac{\sigma-2}{2\sigma-3}}u\big \|_{L^{q,2}_{t}L^{p}_{x}} \notag \\ 
 &\lesssim \big\|D_x^{\frac{\sigma-2}{2\sigma-3}}e^{it\partial^2_x}u_0\big \|_{L^{q,2}_{t}L^{p}_{x}}+\left\|D_x^{\frac{\sigma-2}{2\sigma-3}}\int_{0}^{t} e^{i(t-s)\partial^2_x}\partial_x\big(|u|^{2\sigma}u\big)(s) ds\right \|_{L^{q,2}_{t}L^{p}_{x}} \notag \\ 
&\lesssim  \|u_0\|_{ H^{\frac{1}{2}}}+\left\|D_x^{\frac{\sigma-2}{2\sigma-3}}\big(|u|^{2\sigma}u_{x}\big)\right \|_{L^{1}_{t}L^{2}_{x}} 
\lesssim  \|u_0\|_{ H^{\frac{1}{2}}}+\|u\|^{2\sigma+1}_{X}.  \label{LorImpr3}
	\end{align}
 Then, collecting  \eqref{LorImpr2} and  \eqref{LorImpr3} implies  \eqref{LorImpr1}. We finish the proof.
	\end{proof}

Next  are the local smoothing effects and maximal function estimates.

\begin{lemma}[Local smoothing effects,  see \cite{Takaoka01}]\label{Katosmoothing} 
 We have
\begin{align}
\big\|D^{\frac{1}{2}}_x e^{it\partial^2_x}u_0 \big\|_{L^{\infty}_{x}L^2_{t}} &\lesssim \|u_0 \|_{L^{2}_{x}}, \label{KatoEsti1} \\
 \left\|D^{\frac{1}{2}}_x\int_{0}^{t} e^{i(t-s)\partial^2_x}F(x, s) ds\right\|_{L^{\infty}_{t}L^2_{x}} &\lesssim \|F \|_{L^{1}_{x}L^2_{t}}, \label{KatoEsti2} \\
 \left\|\partial_{x}\int_{0}^{t} e^{i(t-s)\partial^2_x}F(x, s) ds\right\|_{L^{\infty}_{x}L^2_{t}}& \lesssim \|F \|_{L^{1}_{x}L^2_{t}}. \label{KatoEsti3} 
	\end{align}
\end{lemma}

\begin{lemma}[Maximal function estimates,  see \cite{Takaoka01}]\label{MaxFunEsti}
 We have
\begin{align}
	\big\|D^{-\frac{1}{4}}_x e^{it\partial^2_x}u_0 \big\|_{L^{4}_{x}L^{\infty}_{t}} &\lesssim \|u_0 \|_{L^{2}_{x}}, \label{MaxFunEsti1}\\
 \left\|D^{-\frac{1}{4}}_x\int_{0}^{t} e^{i(t-s)\partial^2_x}F(x, s) ds\right\|_{L^{\infty}_{t}L^{2}_{x}} &\lesssim \|F \|_{L^{\frac{4}{3}}_{x}L^1_{t}},\label{MaxFunEsti2}\\
 \left\|D^{-\frac{1}{2}}_x\int_{0}^{t} e^{i(t-s)\partial^2_x}F(x, s) ds\right\|_{L^{4}_{x}L^{\infty}_{t}} &\lesssim \|F \|_{L^{\frac{4}{3}}_{x}L^1_{t}}.\label{MaxFunEsti3} 
	\end{align}
\end{lemma}

\begin{lemma} [Interpolation estimates] \label{InterEsti}
Let $2\leq p,q,\tilde{p}, \tilde{q} \leq \infty$ and $ \frac{4}{p}+\frac{2}{q}=\frac{4}{\tilde{p}}+\frac{2}{\tilde{q}}=1$. Then, we have
	\begin{equation}
		\big\|D_x^{\frac{1}{2}-\frac{3}{p}}e^{it\partial^2_x}u_0\big\|_{L_{x}^{p} L_{t}^{q}} \lesssim  \|u_0\|_{L^{2}_x},  \label{InterEsti1}
	\end{equation}
and 
	\begin{equation}
		\left\|D_x^{1-\frac{3}{p }-\frac{3}{\tilde{p}}}\int_{0}^{t} e^{i(t-s)\partial^2_x}F(x, s)  ds\right\|_{L_{x}^{p} L_{t}^{q}} \lesssim  \|F\|_{L_{x}^{\tilde{p}'} L_{t}^{\tilde{q}'}},  \label{InterEsti1a}
	\end{equation}
where  $\frac{1}{\tilde{p}}+\frac{1}{\tilde{p}'}=\frac{1}{\tilde{q}}+\frac{1}{\tilde{q}'}=1$.

 Taking $p=\frac{4(2+\varepsilon)}{\varepsilon}$ and $q=2+\varepsilon$ in \eqref{InterEsti1} for some $0<\varepsilon\ll 1$ yields that 
	\begin{equation}
		\big\|\partial_x e^{it\partial^2_x}u_0\big\|_{L_{x}^{\frac{4(2+\varepsilon)}{\varepsilon}} L_{t}^{2+\varepsilon}}\lesssim  \|u_0\|_{H^{\frac{1}{2}+\frac{3\varepsilon}{4(2+\varepsilon)}}}.  \label{InterEsti2}
	\end{equation}
In particular, assume that  $u(t)$ is the global solution to \eqref{gDNLS} given in Lemma  \ref {BWX20Result}  with initial data $u_0\in H^{\frac{1}{2}+\frac{3\varepsilon}{4(2+\varepsilon)}}(\mathbb{R})$, then
	\begin{equation}
		 \|\partial_x u \|_{L_{x}^{\frac{4(2+\varepsilon)}{\varepsilon}} L_{t}^{2+\varepsilon}} <  \infty.  \label{InterEsti2b}
	\end{equation}
More generally, if $u(t)$ is the global solution to \eqref{gDNLS} with  small initial data $\|u_0\|_{H^{s}}$ for $\frac{1}{2}\leq s\leq 1$, then 
	\begin{equation}
 \big\|D_x^{\frac{1}{2}-\frac{3}{p}+s}u\big\|_{L_{x}^{p} L_{t}^{q}} \lesssim  \|u_0\|_{H^{s}_x},  \label{InterEsti3}
	\end{equation}
where $ \frac{4}{p}+\frac{2}{q}=1$.
\end{lemma}

\begin{proof}
Using the method of Kenig, Ponce and Vega in \cite{KPV93}, one can  derive \eqref{InterEsti1}  by interpolating \eqref{KatoEsti1} and \eqref{MaxFunEsti1}. 	\eqref{InterEsti1a} is dual to	\eqref{InterEsti1}. It suffices to show   \eqref{InterEsti2b}, as \eqref{InterEsti3} can be obtained using a similar argument. 
	
Applying  \eqref{InterEsti2}, \eqref{InterEsti1a} with   $(\tilde{p}, \tilde{q})=(\infty, 2)$ and   Lemma  \ref {BWX20Result} to the Duhamel's formula of \eqref{gDNLS} deduces  that
\begin{align}
\|\partial_x u\|_{L_x^{\frac{4(2+\varepsilon)}{\varepsilon}}L_t^{2+\varepsilon}}&\lesssim  \|\partial_x e^{it\partial^2_x}u_0\|_{L_x^{\frac{4(2+\varepsilon)}{\varepsilon}}L_t^{2+\varepsilon}}+\left\|\partial_x\int_0^t e^{i(t-s)\partial^2_x}\partial_{x}\big(|u|^{2\sigma}u\big)\right\|_{L_x^{\frac{4(2+\varepsilon)}{\varepsilon}}L_t^{2+\varepsilon}}\notag\\ 
&\lesssim  \|u_0\|_{\dot{H}_x^{\frac{1}{2}+\frac{3\varepsilon}{4(2+\varepsilon)}}}+\left\|D_x^{\frac{3\varepsilon}{4(2+\varepsilon)}}\big(|u|^{2\sigma}\partial_x u\big)\right\|_{L_x^{1}L_t^{2}}.\label{InterEsti4}
\end{align}
The last term in \eqref{InterEsti4} was bounded by $\|u\|^{2\sigma+1}_{X}$, see \eqref{BWX20Result22222a} with $s=\frac{1}{2}+\frac{3\varepsilon}{4(2+\varepsilon)}$.
From Lemma \ref{BWX20Result}, one has
\begin{align}
\|\partial_x u\|_{L_x^{\frac{4(2+\varepsilon)}{\varepsilon}}L_t^{2+\varepsilon}}
&\lesssim  \|u_0\|_{\dot{H}_x^{\frac{1}{2}+\frac{3\varepsilon}{4(2+\varepsilon)}}}+\|u\|^{2\sigma+1}_{X} \notag\\
& \lesssim  \|u_0\|_{ H _x^{\frac{1}{2}+\frac{3\varepsilon}{4(2+\varepsilon)}}}+\|u_0\|_{ H _x^{\frac{1}{2}+\frac{3\varepsilon}{4(2+\varepsilon)}}}^{2\sigma+1} <\infty. \nonumber
\end{align}
This finishes the proof.
\end{proof}

\section{Dispersive decay for gDNLS when $\sigma\geq 2$}              
In this section, we use  global a prior estimates, Lorentz-Strichartz estimates and local smoothing estimates to show the dispersive decay for $\sigma\geq 2$.

{\bf{Proof of Theorem \ref{MainResult1}}}  For $T\in (0, \infty]$, denote
$$\|u \|_{X(T)}:=\sup_{t\in [0, T]}|t|^{\frac{1}{2}}\| u(t)\|_{L^{\infty}_{x}}.$$

By Duhamel's principle, the solution to \eqref{gDNLS} can be written as
\begin{align}
u(t)&=e^{it\partial^2_x}u_0+\int_{0}^{t} e^{i(t-s)\partial^2_x}\partial_{x}\big(|u|^{2\sigma}u\big)(s) ds \notag \\
&=e^{it\partial^2_x}u_0+\int_{0}^{\frac{t}{2}} e^{i(t-s)\partial^2_x}\partial_{x}\big(|u|^{2\sigma}u\big)(s) ds +\int_{\frac{t}{2}}^{t} e^{i(t-s)\partial^2_x}\partial_{x}\big(|u|^{2\sigma}u\big)(s) ds . \label{gDNLSsol}
	\end{align}
Using dispersive decay estimate \eqref{dispEsti} with   $r=\infty$,  we see that the first term on RHS\eqref{gDNLSsol}
can be controlled by:
\begin{align} \|e^{it\partial^2_x}u_0\|_{L^{\infty}_x}\lesssim t^{-\frac{1}{2}}\|u_0\|_{L^{1}_x}. \label{DispDecayLinear}
\end{align}

Next, we estimate the last two term on RHS\eqref{gDNLSsol} respectively. Applying \eqref{dispEsti} with  $r=\infty$,  H\"older inequality \eqref{Holder- Lor1} and Sobolev embedding theorem, we deduce that
\begin{align}
 &\left\| \int_{0}^{\frac{t}{2}}e^{i(t-s)\partial^2_x}\partial_{x}\big(|u|^{2\sigma}u\big)(s) ds  \right\|_{L^{\infty}_{x}}  \notag \\
\lesssim & \int_{0}^{\frac{t}{2}}|t-s|^{-\frac{1}{2}}\big\||u|^{2\sigma}\partial_x u(s) \big\|_{L^{1}_{x}}ds \notag \\
\lesssim&  |t|^{-\frac{1}{2}} \int_{0}^{\frac{t}{2}}\big\||u|^{\frac{4-\varepsilon}{2+\varepsilon}}\partial_x u \big\|_{L^{2+\varepsilon}_{x}}\big\||u|^{2\sigma-\frac{6}{2+\varepsilon}}\big \|_{L^{\frac{2+\varepsilon}{1+\varepsilon}}_{x}}\|u \|_{L^{\infty}_{x}}ds  \notag \\
\lesssim& |t|^{-\frac{1}{2}} \|u \|_{X(T)}\int_{0}^{\frac{t}{2}}s^{-\frac{1}{2}}\big\||u|^{\frac{4-\varepsilon}{2+\varepsilon}}\partial_x u \big\|_{L^{ 2+\varepsilon }_{x}}\big\||u|^{2\sigma-\frac{6}{2+\varepsilon}}\big \|_{L^{\frac{2+\varepsilon}{1+\varepsilon}}_{x}}ds  \notag \\
\lesssim& |t|^{-\frac{1}{2}}\|u \|_{X(T)} \big\||s|^{-\frac{1}{2}}\big\|_{L_s^{2,\infty}} \big\||u|^{\frac{4-\varepsilon}{2+\varepsilon}}\partial_x u \big\|_{L^{ 2+\varepsilon }_{tx}} \big\||u|^{2\sigma-\frac{6}{2+\varepsilon}}\big \|_{L^{\frac{2(2+\varepsilon)}{\varepsilon},\frac{2+\varepsilon}{1+\varepsilon}}_{t}L^{\frac{2+\varepsilon}{1+\varepsilon}}_{x}}  \notag \\
\lesssim&	|t|^{-\frac{1}{2}}\|u \|_{X(T)} \| u\|^{\frac{4-\varepsilon}{2+\varepsilon}}_{L_{x}^{4} L_{t}^{\infty}}\big\|\partial_x u\big\|_{L_{x}^{\frac{4(2+\varepsilon)}{\varepsilon}} L_{t}^{2+\varepsilon}}\big\|u\big \|^{2\sigma-\frac{6}{2+\varepsilon}}_{L^{\frac{4\sigma(2+\varepsilon)-12}{\varepsilon},2}_{t}L^{\frac{2\sigma(2+\varepsilon)-6}{1+\varepsilon}}_{x}} \notag \\
\lesssim&|t|^{-\frac{1}{2}}\|u \|_{X(T)}\| u\|^{\frac{4-\varepsilon}{2+\varepsilon}}_{L_{x}^{4} L_{t}^{\infty}}\big\|\partial_x u\big\|_{L_{x}^{\frac{4(2+\varepsilon)}{\varepsilon}} L_{t}^{2+\varepsilon}}\Big\|J_x^{\frac{(\sigma-2)(2+\varepsilon)}{2\sigma(2+\varepsilon)-6}}u\Big \|^{2\sigma-\frac{6}{2+\varepsilon}}_{L^{\frac{4\sigma(2+\varepsilon)-12}{\varepsilon},2}_{t}L^{\frac{2\sigma(2+\varepsilon)-6}{\sigma(2+\varepsilon)-3-\varepsilon}}_{x}}. \label{DispDecay0Ia2a}
	\end{align}
Note that the last two term in RHS\eqref{DispDecay0Ia2a} are estimated in \eqref{InterEsti2b} and \eqref{LorImpr1} respectively. So, we have
\begin{align}
  \left\| \int_{0}^{\frac{t}{2}}e^{i(t-s)\partial^2_x}\partial_{x}\big(|u|^{2\sigma}u\big)(s) ds  \right\|_{L^{\infty}_{x}}   
\leq 	C(\|u_0\|_{H^{\frac{1}{2}+\frac{\sigma-2}{2\sigma-3}}})|t|^{-\frac{1}{2}}\|u \|_{X(T)}. \label{DispDecay0Ia2}
	\end{align}
Observe that the Lorentz space improvement in \eqref{LorImpr1} was crucial to compensate for the fact that  $|s|^{-\frac{1}{2}}$ lies only in the Lorentz space $L^{2, \infty}_s$.

Arguing similarly,   one gets
\begin{align}
 &\left\| \int_{\frac{t}{2}}^{t}e^{i(t-s)\partial^2_x}\partial_{x}\big(|u|^{2\sigma}u\big)(s) ds  \right\|_{L^{\infty}_{x}}  \notag \\
\lesssim & \int_{\frac{t}{2}}^{t}|t-s|^{-\frac{1}{2}}\big\||u|^{2\sigma}\partial_x u(s) \big\|_{L^{1}_{x}}ds \notag \\
\lesssim&  \int_{\frac{t}{2}}^{t}|t-s|^{-\frac{1}{2}}\big\||u|^{\frac{4-\varepsilon}{2+\varepsilon}}\partial_x u \big\|_{L^{2+\varepsilon}_{x}}\big\||u|^{2\sigma-\frac{6}{2+\varepsilon}}\big \|_{L^{\frac{2+\varepsilon}{1+\varepsilon}}_{x}}\|u \|_{L^{\infty}_{x}}ds  \notag \\
\lesssim& |t|^{-\frac{1}{2}} \|u \|_{X(T)}\int_{\frac{t}{2}}^{t}|t-s|^{-\frac{1}{2}}\big\||u|^{\frac{4-\varepsilon}{2+\varepsilon}}\partial_x u \big\|_{L^{ 2+\varepsilon }_{x}}\big\||u|^{2\sigma-\frac{6}{2+\varepsilon}}\big \|_{L^{\frac{2+\varepsilon}{1+\varepsilon}}_{x}}ds  \notag \\
\lesssim& |t|^{-\frac{1}{2}}\|u \|_{X(T)} \big\||t-s|^{-\frac{1}{2}}\big\|_{L_s^{2,\infty}} \big\||u|^{\frac{4-\varepsilon}{2+\varepsilon}}\partial_x u \big\|_{L^{ 2+\varepsilon }_{tx}} \big\||u|^{2\sigma-\frac{6}{2+\varepsilon}}\big \|_{L^{\frac{2(2+\varepsilon)}{\varepsilon},\frac{2+\varepsilon}{1+\varepsilon}}_{t}L^{\frac{2+\varepsilon}{1+\varepsilon}}_{x}}  \notag \\
\lesssim&	|t|^{-\frac{1}{2}}\|u \|_{X(T)} \| u\|^{\frac{4-\varepsilon}{2+\varepsilon}}_{L_{x}^{4} L_{t}^{\infty}}\big\|\partial_x u\big\|_{L_{x}^{\frac{4(2+\varepsilon)}{\varepsilon}} L_{t}^{2+\varepsilon}}\|u\|^{2\sigma-\frac{6}{2+\varepsilon}}_{L^{\frac{4\sigma(2+\varepsilon)-12}{\varepsilon},2}_{t}L^{\frac{2\sigma(2+\varepsilon)-6}{1+\varepsilon}}_{x}} \notag \\
\leq&	C(\|u_0\|_{H^{\frac{1}{2}+\frac{\sigma-2}{2\sigma-3}}})|t|^{-\frac{1}{2}}\|u \|_{X(T)}  \nonumber
	\end{align}
which together with \eqref{DispDecay0Ia2} derives 
\begin{align}
\sup_{t\in [0, T]}|t|^{\frac{1}{2}}\| u(t)  \|_{L^{\infty}_{x}} \leq  \|u_0 \|_{L^{1}_{x}}+C(\|u_0\|_{H^{\frac{1}{2}+\frac{\sigma-2}{2\sigma-3}}}) \|u \|_{X(T)}. \label{DispDecay0mgDNLSf1}
	\end{align}
As $\|u_0\|_{H^{\frac{1}{2}+\frac{\sigma-2}{2\sigma-3}}}\ll 1$,   \eqref{DispDecay0mgDNLSf1} yields the desired result \eqref{MainResult1a} immediately. So, we complete the proof of this theorem.
 
\section{Global dynamics for gDNLS when $1\leq\sigma< 2$}           
In this section, we apply vector field methods to show dispersive estimate for  small data solution  to  the gDNLS equation \eqref{gDNLS}. 

The vector field used is defined as
$$L:=x+2it\partial_x,$$
which   is the conjugate of $x$ with respect to the linear flow, $e^{it\partial_x^2}x=Le^{it\partial_x^2}$ , as well as the generator for the Galilean group of symmetries. Hence, it is easy to see 
\begin{align}
[i \partial_t+\partial_x^2, L]=0, \hspace{8mm}  [\partial_x, L]=1,  \nonumber
	\end{align}
and
\begin{align}
L(|u|^{2\sigma}u)=(\sigma+1)|u|^{2\sigma}Lu-\sigma |u|^{2\sigma-2}u^{2}\overline{Lu}. \label{VecField}
	\end{align}

Let us explain the vector field method  briefly. A simple computation reveals  
$$\frac{d}{dx}|u|^2=\frac{1}{2it}(\bar{u}Lu-u\overline{Lu})$$
which yields by integration and H\"older inequality that
\begin{align}
\|u\|^2_{L^{\infty}}\lesssim \frac{1}{t}\|u(t)\|_{L^{2}}\|Lu(t)\|_{L^{2}}.
\label{KlainSobolev}
	\end{align}
This is what we call a vector field bound, or a Klainerman-Sobolev type inequality. The most important thing is that the vector field bound no longer depends on the fact that $u$ solves the linear Schr\"odinger equation. Hence, it is easy to be applied to the nonlinear problem.

We see that dispersive decay for the solutions to \eqref{gDNLS} would follow from \eqref{KlainSobolev} as long as we obtain energy estimates for $u$ and $Lu$. However, the estimate for $Lu$ is not uniformly bounded when $\sigma=1$. On the contrary, it exhibits a slow growth in time, see \eqref{gDNLS-Lu3} and \eqref{gDNLS-Lu4}. To overcome this difficulty, Ifrim and Tataru \cite{IfTa15}   developed the testing by wave packets method. Via using a mixed wave packet style phase space localization, instead of localizing sharply on either the Fourier or the physical side, one can better balance the scales so that linear and nonlinear matching errors become comparable. The vector field combining with the testing by wave packets  method has become an important tool for studying the long range behavior for nonlinear dispersive equations. We refer to \cite{IfTa24} for more introduction and application of this method.

\subsection{Wave packets and the asymptotic equation}
In order to study the global decay properties of solutions to \eqref{gDNLS} when the nonlinear effect is strong, we use wave packets traveling along the Hamilton flow to test the solution.  A wave packet $\Phi_v$ is an approximate solution to the linear equation, with $O(1/t)$ errors, see \eqref{WavePacket1a1}. For each trajectory $\Gamma_v:=\left\{x=vt\right\}$ with velocity $v$, we will show dispersive decay along this ray by testing with a wave packet moving along the ray. To be precise,  an asymptotic profile $\gamma$ to the  solution $u$ will be constructed by taking the $L^2$ inner product of $\Phi_v$ with  $u$. The asymptotic profile $\gamma$ provides a good description of $u$ as $t\to \infty$. In fact, Lemma \ref{DifferenceBounds} quantitatively shows us the relationship between $u$ along the ray $\Gamma_v$ and $\gamma$. Besides, we found the approximate ODE dynamics for $\gamma$, which was called the asymptotic equation, see Lemma \ref{AsympEq}.  Consequently, the upper bounds for the asymptotic profile $\gamma$ can be obtained  via the asymptotic equation which is much easier to be studied. Then one can establish the bounds for $u$ with the help of these bounds for $\gamma$.

Now, we recall the wave packet for the Schr\"odinger flow. The phase function is given by 
$$\phi(t,x)=\frac{x^2}{4t},$$
which is the same as the phase of the fundamental solution to the linear Schr\"odinger equation. From the phase function, we see that the spatial frequency corresponding to the ray $\Gamma_v$ is
$\xi_v:=\frac{x}{2t}=\frac{v}{2}.$
Define the wave packet along the ray $\Gamma_v$ by
\begin{align}
\Phi_v=e^{i\phi(t,x)}\chi\left(\frac{x-vt}{\sqrt{t}}\right),
\label{WavePacket}
	\end{align}
where $\chi$ is a smooth function supported on $[-1, 1]$ and $\int\chi(y)dy=1$. We note that, to stay coherent on dyadic time scales, the spatial scale for localization chosen in the wave packet is $\sqrt{t}$. 

A direct computation implies 
\begin{align}
(i\partial_t+\partial_x^2)\Phi_v=\frac{e^{i\phi(t,x)}}{2t}\partial_x \left[i(x-vt)\chi\left(\frac{x-vt}{\sqrt{t}}\right)+2\sqrt{t}\chi'\left(\frac{x-vt}{\sqrt{t}}\right)\right],
\label{WavePacket1a1}
	\end{align}
which suggests that $\Phi_v$ is indeed a good approximation of the solution to the linear Schr\"odinger equation.

In order to measure better the decay of $u$ along  $\Gamma_v$, we use the following asymptotic profile defined by
\begin{align}
 \gamma(t, v):=\int u(t,x)\overline{\Phi_v}(t,x)dx.
\label{AsymProfile}
	\end{align}
We make some useful observations about $\gamma$. Define $w(t,x):=e^{-i\phi(t,x)}u(t,x)$. Then we have 
\begin{align}
\partial_x w=\frac{-i}{2t}e^{-i\phi}Lu\label{wFunc}
	\end{align}
and 
\begin{align}
 \gamma(t, v)=\int w(t,x)\overline{\chi\left(\frac{x-vt}{\sqrt{t}}\right)}dx=t^{\frac{1}{2}}w(t,vt)*_{v}t^{\frac{1}{2}}\chi(vt^{\frac{1}{2}}).
\label{AsymProfile1a}
	\end{align}
We can also express $\gamma$ in terms of the Fourier transform of $u$ by Plancherel's identity,
$$\gamma(t, v)=\int \hat{u}(t,\xi)\overline{\widehat{\Phi}_v}(t,\xi)d\xi.$$
By a direct computation, we get
$$\widehat{\Phi}_v(t,\xi)=t^{\frac{1}{2}}e^{-it\xi^2}\chi_1\left(\frac{(2\xi-v)t^{\frac{1}{2}}}{2}\right),$$
where $\chi_1(\xi)=e^{i\xi^2}\mathscr{F}\left(e^{iy^2/2}\chi(y)\right)(\xi)$ with the property that
$$\int\chi_1(\xi)d\xi \approx \int\chi(y)dy=1.$$
Hence,  we can write
\begin{align}
 \gamma(t, \xi)= \frac{1}{2}e^{\frac{it\xi^2}{4}} \widehat{u}\big(t,\frac{\xi}{2}\big)*_{\xi}t^{\frac{1}{2}}\chi_1\left(\frac{t^{\frac{1}{2}}\xi}{2}\right).
\label{AsymProfile1b}
	\end{align}

The following Lemma shows that the asymptotic profile $\gamma(t, v)$ is compared to the solution $u$ along the ray $\Gamma_v$ and $\widehat{u}$ evaluated at $v$.
\begin{lemma} \label{DifferenceBounds}
the asymptotic profile $\gamma$ defined in \eqref{AsymProfile} satisfies the bounds
	\begin{align}
\|\gamma\|_{L^{\infty}}\lesssim   t^{\frac{1}{2}} \|u\|_{L^{\infty}}, \hspace{8mm} 
\|\gamma\|_{L^{\infty}}\lesssim   \|Lu\|_{L^{2}_x}+ \|u\|_{L^{2}_x}, \label{DifferenceBounds1a} \\
\|\gamma\|_{L^{2}_v}\lesssim     \|u\|_{L^{2}_x}, \hspace{8mm}
\|\partial_v\gamma\|_{L^{2}_v}\lesssim   \|Lu\|_{L^{2}_x}, \label{DifferenceBounds1aa}
	\end{align}
and 
	\begin{align}
&\hspace{22.3mm} \|v\gamma\|_{L^{2}_v}\lesssim  t^{-1}\|Lu\|_{L^{2}_x}+\|u_x\|_{L^{2}_x}, \label{DifferenceBounds1b1} \\
&\| v\gamma\|_{L^{\infty}} \lesssim t^{-\frac{1}{2}}\left(\|Lu\|_{L^{2}_x}+\|u\|_{L^{2}_x}+\|Lu_x\|_{L^{2}_x} \right)+t^{\frac{1}{2}}\|u_x\|_{L^{\infty}}. \label{DifferenceBounds1b2}
	\end{align}
We have the physical space difference bounds
	\begin{align}
\|u(t, vt)-t^{-\frac{1}{2}}e^{i\phi(t,vt)}\gamma(t, v)\|_{L^{\infty}}\lesssim   t^{-\frac{3}{4}}\|Lu\|_{L^{2}_x}, \label{DifferenceBounds2a}\\
\| u(t, vt)-t^{-\frac{1}{2}}e^{i\phi(t,vt)}\gamma(t, v)\|_{L^{2}_v}\lesssim  t^{-1} \|Lu\|_{L^{2}_x}, \label{DifferenceBounds2b}
	\end{align}
and
	\begin{align}
\|u_x(t, vt)-\frac{i}{2}t^{-\frac{1}{2}}e^{i\phi(t,vt)}v\gamma(t, v)\|_{L^{\infty}}\lesssim   t^{-\frac{3}{4}}\left(\|Lu\|_{L^{2}_x}+\|Lu_x\|_{L^{2}_x}\right), \label{DifferenceBounds2c}
	\end{align}
and the Fourier space bounds
	\begin{align}
\|\widehat{u}(t, \xi)- e^{-it\xi^2}\gamma(t, 2\xi)\|_{L^{\infty}}\lesssim   t^{-\frac{1}{4}}\|Lu\|_{L^{2}_x}, \label{DifferenceBounds3a}\\
\| \widehat{u}(t, \xi)- e^{-it\xi^2}\gamma(t, 2\xi)\|_{L^{2}_{\xi}}\lesssim  t^{-\frac{1}{2}} \|Lu\|_{L^{2}_x}.\label{DifferenceBounds3b} 
	\end{align}
\end{lemma}
\begin{remark}
This lemma indicates $u$ is approximated well by $\gamma$  both on the physical side and the Fourier side. We note that $u$ may not be the solution to \eqref{gDNLS}.
\end{remark}
\begin{remark}
\eqref{DifferenceBounds1a}-\eqref{DifferenceBounds1aa}, \eqref{DifferenceBounds2a}-\eqref{DifferenceBounds2b} and \eqref{DifferenceBounds3a}-\eqref{DifferenceBounds3b} had been proved by Ifrim and Tataru in \cite{IfTa15} where the cubic nonlinear Schr\"odinger equation was consider.  When it comes to derivative nonlinear Schr\"odinger equation (with $\sigma=1$ in \eqref{gDNLS}), we need to control $\|v\gamma\|_{L^{2}_v}$, $\| v\gamma\|_{L^{\infty}}$ and $\| u_x\|_{L^{\infty}}$. Byars \cite{Bya24} showed   \eqref{DifferenceBounds2c} and the following analogue of \eqref{DifferenceBounds1b1} and \eqref{DifferenceBounds1b2}
	\begin{align}
\| \langle v\rangle^k\gamma\|_{L^{2}_v}\lesssim   \|u\|_{H^{k}_x}, \hspace{7mm}
\| \langle v\rangle^{\frac{k}{2}} \gamma\|_{L^{\infty}}\lesssim   \|Lu\|_{L^{2}_x}+\|u\|_{H^{k}_x}. \label{DifferenceBounds1bS}
	\end{align}
\eqref{DifferenceBounds1bS} played an important role in estimating the remainder term for the asymptotic equation of $\gamma$ in \cite{Bya24}. 
\end{remark}
\begin{proof}
As mentioned above, it suffices to prove \eqref{DifferenceBounds1b1} and \eqref{DifferenceBounds1b2}. Without loss of generality, we may assume that $v>1$. Note that for $\chi$ supported on $[-1, 1]$, $\chi\left(\frac{x-vt}{\sqrt{t}}\right)$ supported on $I:=[-\sqrt{t}+vt, \sqrt{t}+vt]$. Hence, we have  $|x|\approx |vt|$ in $I$. 

It follows from \eqref{AsymProfile1a} and Young inequality that
	\begin{align}
\| v\gamma\|_{L^{2}_v} \lesssim &t^{\frac{1}{2}}\|vw(t,vt)\|_{L^{2}_v}\|t^{\frac{1}{2}}\chi(vt^{\frac{1}{2}})\|_{L^{1}_v} \notag \\
\lesssim & t^{-1}\|xu(t,x)\|_{L^{2}_x}\lesssim t^{-1}\|Lu\|_{L^{2}_x}+\|u_x\|_{L^{2}_x}, \nonumber
	\end{align}
which gives the desired estimate \eqref{DifferenceBounds1b1}.

Using the identity \eqref{AsymProfile1a} again implies
	\begin{align}
\| v\gamma\|_{L^{\infty}} \lesssim &t^{\frac{1}{2}}\|vw(t,vt)\|_{L^{\infty}}\|t^{\frac{1}{2}}\chi(vt^{\frac{1}{2}})\|_{L^{1}_v} \notag \\
\lesssim & t^{-\frac{1}{2}}\|xu(t,x)\|_{L^{\infty}}\lesssim t^{-\frac{1}{2}}\left(\|Lu\|_{L^{\infty}}+t\|u_x\|_{L^{\infty}}\right). \label{DiffBoundPf2a}
	\end{align}
By taking advantage of Gagliardo-Nirenberg inequality and the fact $[\partial_x, L]=1$, we get
	\begin{align}
\|Lu\|_{L^{\infty}}\lesssim \|Lu\|_{L^{2}_x}+\|\partial_xLu\|_{L^{2}_x}\lesssim \|Lu\|_{L^{2}_x}+\|u\|_{L^{2}_x}+\|Lu_x\|_{L^{2}_x}. \label{DiffBoundPf2b}
	\end{align}
Substituting \eqref{DiffBoundPf2b} into \eqref{DiffBoundPf2a} yields \eqref{DifferenceBounds1b2}.

 So, we finish the proof of this lemma.\end{proof}

In the next lemma, we find an approximate ODE dynamics for $\gamma(t, v)$. Then, by utilizing the asymptotic equation, we can obtain uniform bounds for $\gamma$ as the time goes to infinity.
	\begin{lemma}\label{AsympEq}
Let $\sigma\geq \frac{1}{2}$. If $u$  solves \eqref{gDNLS} then we have
	\begin{align}
i\gamma_t(t)=\frac{vt^{-\sigma}}{2}|\gamma|^{2\sigma}\gamma-R(t, v), \label{AsympEq1}
	\end{align}
where the remainder $R(t, v)$ satisfies
	\begin{align}
\|R\|_{L^{\infty}}\lesssim & t^{-\frac{5}{4}}\|Lu\|_{L_x^2}+ \|u\|_{L^{\infty}}^{2\sigma+1}+  t^{-\frac{1}{4}}\|u\|_{L^{\infty}}^{2\sigma-1}\|u_x\|_{L^{\infty}}\|Lu\|_{L^2_x}\notag \\
& +t^{-\frac{5}{4}}\|u\|_{L^{\infty}}^{2\sigma-1}\|Lu\|_{L^2_x}\big(\|Lu\|_{L^{2}_x}+\|u\|_{L^{2}_x}+\|Lu_x\|_{L^{2}_x}\big), \label{AsympEq2}
	\end{align}
\begin{align}
\|R\|_{L^{2}_v}\lesssim& 
t^{-\frac{3}{2}}\|Lu\|_{L_x^2}+ t^{-\frac{1}{2}}\|u\|_{L^{\infty}}^{2\sigma}\|u\|_{L^{2}} 
+  t^{-\frac{1}{2}}  \|u\|^{2\sigma-1}_{L^{\infty}}\|u_x\|_{L^{\infty}}
 \|Lu\|_{L^{2}_x}\notag \\
&+  t^{-\frac{3}{2}}\|u\|_{L^{\infty}}^{2\sigma-1}\|Lu\|_{L^2_x}\big(\|Lu\|_{L^{2}_x}+\|u\|_{L^{2}_x}+\|Lu_x\|_{L^{2}_x}\big), \label{AsympEq2-L2}
	\end{align}
and 
	\begin{align}
\|vR\|_{L^{\infty}} \lesssim&
t^{-1}\|u\|_{L^{\infty}_x}+t^{-\frac{7}{4}}\|Lu\|_{L^2_x}+t^{-\frac{5}{4}}\|Lu_x\|_{L^2_x} +t^{-\frac{5}{4}}\|u\|^{2\sigma}_{L^{\infty}_x} \|Lu\|_{L^{2}_x}  \notag \\
& + \|u\|^{2\sigma}_{L^{\infty}_x}\|u_x\|_{L^{\infty}_x} +t^{-\frac{7}{4}} \|u\|^{2\sigma-1}_{L^{\infty}_x}\|Lu\|_{L^2_x} \big(\|Lu\|_{L^{2}_x}+\|u\|_{L^{2}_x}+\|Lu_x\|_{L^{2}_x}\big)\notag \\
 &+t^{-\frac{5}{4}} \|u\|^{2\sigma-1}_{L^{\infty}_x}\big( \|Lu\|_{L_x^2}+\|u\|_{L_x^2}+\|Lu_x\|_{L_x^2}\big)^2\notag \\
&+t^{-\frac{1}{4}} \|u\|^{2\sigma-1}_{L^{\infty}_x}\big( \|Lu\|_{L_x^2}+\|u\|_{L_x^2}+\|Lu_x\|_{L_x^2}\big)\|u_x\|_{L^{\infty}_x} .      \label{AsympEq3}
	\end{align}
	\end{lemma}
\begin{remark}
For derivative nonlinear Schr\"odinger equation (with $\sigma=1$ in \eqref{gDNLS}), Byars \cite{Bya24} obtained  the following upper bounds  
	\begin{align}
\|R\|_{L^{\infty}}\lesssim  &t^{-\frac{5}{4}}\|Lu\|_{L^{2}_x}+\|u\|^3_{L^{\infty}}+t^{-\frac{1}{4}}\|u\|^2_{L^{\infty}}\big(\|Lu\|_{L^{2}_x}+\|u\|_{H^{3}_x}\big)\notag \\
&  +t^{-\frac{1}{4}}\|u\|^2_{L^{\infty}}\|Lu\|_{L^{2}_x}\big(\|Lu\|_{L^{2}_x}+\|u\|_{H^{2}_x}\big),   \label{AsympEq2S}
	\end{align}
and
	\begin{align}
\|vR\|_{L^{\infty}}\lesssim  &t^{-\frac{13}{12}}\big(\|Lu\|_{L^{2}_x}+\|u\|_{H^{3}_x}\big)+t^{-\frac{5}{4}}\|u\|^2_{L^{\infty}}\|Lu\|_{L^{2}_x}\notag \\
& +\|u\|^2_{L^{\infty}}\|u_x\|_{L^{\infty}} +t^{-\frac{3}{20}}\|u\|^2_{L^{\infty}}\big(\|Lu\|_{L^{2}_x}+\|u\|_{H^{5}_x}\big) \notag \\
&+t^{-\frac{3}{4}}\|u\|_{L^{\infty}}\|Lu\|_{L^{2}_x}\big(\|Lu\|_{L^{2}_x}+\|u\|_{H^{4}_x}\big).  \label{AsympEq3S}
	\end{align}
Comparing  \eqref{AsympEq2}-\eqref{AsympEq3} with \eqref{AsympEq2S}-\eqref{AsympEq3S}, we have significantly lowered the regularity. Moreover, we obtain better decay estimates with respect to time $t$.
\end{remark}
\begin{proof} A direct computation yields
\begin{align}
 i\gamma_t=\int iu_t\overline{\Phi_v}+iu\partial_t\overline{\Phi_{v}}dx=\int -u\overline{(i\partial_t+\partial^2_x)\Phi_v}-i\partial_x\big(|u|^{2\sigma}u\big)\overline{\Phi_{v}}dx.
\nonumber
	\end{align}
According to identity \eqref{WavePacket1a1} and \eqref{wFunc}, we have
\begin{align}
 i\gamma_t&=\int -u\frac{e^{-i\phi}}{2t}\partial_x \left[-i(x-vt)\chi\left(\frac{x-vt}{\sqrt{t}}\right)+2\sqrt{t}\chi'\left(\frac{x-vt}{\sqrt{t}}\right)\right]+i\big(|u|^{2\sigma}u\big)\partial_x\overline{\Phi_{v}}dx \notag\\ 
&=\int \frac{-ie^{-i\phi}}{4t^2}Lu \left[-i(x-vt)\chi\left(\frac{x-vt}{\sqrt{t}}\right)+2\sqrt{t}\chi'\left(\frac{x-vt}{\sqrt{t}}\right)\right]+i\big(|u|^{2\sigma}u\big)\partial_x\overline{\Phi_{v}}dx.
\nonumber
	\end{align}

Note that 
$$\partial_x \Phi_{v}=t^{-\frac{1}{2}}\Psi_v+\frac{ix}{2t}\Phi_{v},$$
where $\Psi_v:=e^{i\phi}\chi'(\cdot)$ is very similar to  $\Phi_{v}$. Hence, we can write
\begin{align}
 i\gamma_t=\frac{1}{2}t^{-\sigma}v|\gamma|^{2\sigma}\gamma-R(t,v),
\nonumber
	\end{align}
where the remainder term is 
\begin{align}
R(t,v)=&\int \frac{ie^{-i\phi}}{4t^2}Lu \left[-i(x-vt)\chi\left(\frac{x-vt}{\sqrt{t}}\right)+2\sqrt{t}\chi'\left(\frac{x-vt}{\sqrt{t}}\right)\right]dx \notag \\
&-i\int\big(|u|^{2\sigma}u\big)\partial_x\overline{\Phi_{v}}dx+\frac{1}{2}t^{-\sigma}v|\gamma|^{2\sigma}\gamma.
\nonumber
	\end{align}
In order to show \eqref{AsympEq2} and \eqref{AsympEq3},  We split $R(t,v)$ into four pieces
$$R(t,v)=R_1+R_2+R_3+R_4$$
with
\begin{align}
R_1=&\int \frac{ie^{-i\phi}}{4t^2}Lu \left[-i(x-vt)\chi\left(\frac{x-vt}{\sqrt{t}}\right)+2\sqrt{t}\chi'\left(\frac{x-vt}{\sqrt{t}}\right)\right]dx, \notag \\
R_2=&-i\int\big(|u|^{2\sigma}u\big)t^{-\frac{1}{2}}\overline{\Psi_v}dx,\notag \\
R_3=&-\int\frac{vu}{2}\overline{\Phi_{v}}\big(|u|^{2\sigma}-|u(t, vt)|^{2\sigma}\big)dx,\notag \\
R_4=&\frac{v\gamma}{2}\big(t^{-\sigma}|\gamma|^{2\sigma}-|u(t, vt)|^{2\sigma}\big).
\nonumber
	\end{align}

\vspace{1mm}
\noindent
$\bullet$ {\bf Estimate for  $\|R\|_{L^{\infty}}$ and $\|R\|_{L_v^{2}}$}  
\vspace{1mm}

By using H\"older inequality, one can bound $R_1$  
\begin{align}
|R_1|&\lesssim t^{-2}\|Lu\|_{L_x^2}\left( \left\|(x-vt)\chi\left(\frac{x-vt}{\sqrt{t}}\right) \right\|_{L_x^2}+ \sqrt{t}\left\|\chi'\left(\frac{x-vt}{\sqrt{t}}\right) \right\|_{L_x^2} \right) \notag\\
&\lesssim t^{-2}\|Lu\|_{L_x^2}\left( t^{\frac{3}{4}}\|x\chi(x)\|_{L_x^2}+ t^{\frac{3}{4}}\|\chi'(x)\|_{L_x^2}\right)\lesssim t^{-\frac{5}{4}}\|Lu\|_{L_x^2}.
\label{Remainder111R_1}
	\end{align}
Note that $R_1$ can be expressed as a convolution in $v$,
$$R_1=\frac{1}{2t}  \left(itv\chi(t^{\frac{1}{2}}v)-2\sqrt{t}\chi'(t^{\frac{1}{2}}v)\right)*_v \left(\partial_v w(t, tv)\right).$$
So we have
\begin{align}
\|R_1\|_{L_v^2}&\lesssim t^{-1}\left( \| tv\chi(t^{\frac{1}{2}}v) \|_{L_v^1}+  \|\sqrt{t}\chi'(t^{\frac{1}{2}}v)\|_{L_v^1} \right) \|\partial_v w(t, tv)\|_{L_v^2}\lesssim t^{-\frac{3}{2}}\|Lu\|_{L_x^2}.
\label{Remainder111R_1-L2}
	\end{align}

Next, we bound $R_2$ as follows:
\begin{align}
|R_2|\lesssim \|u\|_{L^{\infty}}^{2\sigma+1}t^{-\frac{1}{2}}\int\left|\overline{\Psi_v}\left(\frac{x-vt}{\sqrt{t}}\right)\right|dx\lesssim \|u\|_{L^{\infty}}^{2\sigma+1},
\label{Remainder111R_2}
	\end{align}
and 
\begin{align}
\|R_2\|_{L_v^2}&=\left\|(|u|^{2\sigma}u)(tv)*_vt^{\frac{1}{2}}\overline{\Psi_v}(\sqrt{t}v)\right\|_{L_v^2} \notag\\
&\lesssim \left\|(|u|^{2\sigma}u)(tv)\right\|_{L_v^2} \|t^{\frac{1}{2}}\overline{\Psi_v}(\sqrt{t}v)\|_{L_v^1} 
\lesssim t^{-\frac{1}{2}}\|u\|_{L^{\infty}}^{2\sigma}\|u\|_{L^{2}}.
\label{Remainder111R_2-L2}
	\end{align}

Let us estimate $R_3$. We observe that $|x|\approx |vt|$ from the support of $\chi$, hence 
\begin{align}
|R_3|&\lesssim t^{-1}\int\frac{|xu\overline{\Phi_{v}}|}{2}\big||u|^{2\sigma}-|u(t, vt)|^{2\sigma}\big|dx \notag \\
&\lesssim t^{-1} \|u\|_{L^{\infty}}^{2\sigma-1}\int\big|(Lu-2itu_x)\overline{\Phi_{v}}\big|\big|w(t, x)-w(t, vt) \big|dx \notag \\
&\lesssim t^{-1} \|u\|_{L^{\infty}}^{2\sigma-1}\int\big|Lu\overline{\Phi_{v}} \big|\big|w(t, x)-w(t, vt) \big|dx\notag \\
&\hspace{5mm}+\|u\|_{L^{\infty}}^{2\sigma-1}\int\big|u_x \overline{\Phi_{v}}\big|\big|w(t, x)-w(t, vt) \big|dx\notag \\
&:=R_{3,I}+R_{3,II}
\label{Remainder111R_3}
	\end{align}
provided $\sigma\geq \frac{1}{2}$.

 By the change of variables $x=vt+zt$, we have
\begin{align}
R_{3,I}&\lesssim t^{-1} \|u\|_{L^{\infty}}^{2\sigma-1}\|Lu\|_{L^{\infty}}\int\chi\left(\frac{x-vt}{\sqrt{t}}\right) \big|w(t, x)-w(t, vt) \big|dx\notag \\
&\lesssim  \|u\|_{L^{\infty}}^{2\sigma-1}\|Lu\|_{L^{\infty}}\int\chi(\sqrt{t}z) \big|w(t, vt+zt)-w(t, vt) \big|dz. 
\label{Remainder111R_3Ia1}
	\end{align}
By the Fundamental Theorem of Calculus,
\begin{align}
\big|w(t, (v+z)t)-w(t, vt) \big|\lesssim \int_{v}^{z+v}| \partial_yw(t, yt)|dy\lesssim |z|^{\frac{1}{2}}\|\partial_vw(t, vt)\|_{L^2_v}
\nonumber
	\end{align}
which yields
\begin{align}
\int\chi(\sqrt{t}z) \big|w(t, vt+zt)-w(t, vt) \big|dz \lesssim \big\||z|^{\frac{1}{2}}\chi(\sqrt{t}z)\big\|_{L^1_z} \big\|\partial_vw(t, vt)\big\|_{L^2_v}.
\nonumber
	\end{align}
Notice that 
$$\partial_vw(t, vt)=t\big(\partial_xe^{-i\phi}u)(t, vt)=-\frac{i}{2}\big(e^{-i\phi}Lu\big)(t, vt).$$
So, 
\begin{align}
\int\chi(\sqrt{t}z) \big|w(t, vt+zt)-w(t, vt) \big|dz
\lesssim t^{-\frac{3}{4}} \big\|Lu(t, vt)\big\|_{L^2_v}\lesssim t^{-\frac{5}{4}}\|Lu\|_{L^2_x}.
\label{Remainder111R_3Ia2}
	\end{align}
Then, it follows from \eqref{Remainder111R_3Ia1},\eqref{Remainder111R_3Ia2} and \eqref{DiffBoundPf2b} that
\begin{align}
R_{3,I}&\lesssim  t^{-\frac{5}{4}}\|u\|_{L^{\infty}}^{2\sigma-1}\|Lu\|_{L^2_x} \|Lu\|_{L^{\infty}}\notag \\
&\lesssim  t^{-\frac{5}{4}}\|u\|_{L^{\infty}}^{2\sigma-1}\|Lu\|_{L^2_x}\big(\|Lu\|_{L^{2}_x}+\|u\|_{L^{2}_x}+\|Lu_x\|_{L^{2}_x}\big).
\label{Remainder111R_3Ia3}
	\end{align}

Arguing similarly, one gets
\begin{align}
R_{3,II}&\lesssim  \|u\|_{L^{\infty}}^{2\sigma-1}\|u_x\|_{L^{\infty}}\int\chi\left(\frac{x-vt}{\sqrt{t}}\right) \big|w(t, x)-w(t, vt) \big|dx\notag \\
&\lesssim  t^{-\frac{1}{4}}\|u\|_{L^{\infty}}^{2\sigma-1}\|u_x\|_{L^{\infty}}\|Lu\|_{L^2_x}.
\label{Remainder111R_3IIa1}
	\end{align}
 Putting \eqref{Remainder111R_3Ia3} and \eqref{Remainder111R_3IIa1} into \eqref{Remainder111R_3}, we have
\begin{align}
|R_{3}|\lesssim & t^{-\frac{5}{4}}\|u\|_{L^{\infty}}^{2\sigma-1}\|Lu\|_{L^2_x}\big(\|Lu\|_{L^{2}_x}+\|u\|_{L^{2}_x}+\|Lu_x\|_{L^{2}_x}\big)\notag \\
& +  t^{-\frac{1}{4}}\|u\|_{L^{\infty}}^{2\sigma-1}\|u_x\|_{L^{\infty}}\|Lu\|_{L^2_x}.
\label{Remainder111R_3f}
	\end{align}

To show the $L_v^2$ bound of $R_3$, we use the following estimate
\begin{align}
\big|w(t, (v+z)t)-w(t, vt) \big|\lesssim \int_{0}^{1}|z \partial_vw(t, (v+hz)t)|dh
\nonumber
	\end{align}
 to obtain
\begin{align}
\|R_3\|_{L_v^2}&\lesssim  \|u\|_{L^{\infty}}^{2\sigma-1}\left(\|Lu\|_{L^{\infty}}+t\|u_x\|_{L^{\infty}}\right)\left\|\int\chi(\sqrt{t}z) \big|w(t, vt+zt)-w(t, vt) \big|dz\right\|_{L_v^2}\notag \\
&\lesssim  \|u\|_{L^{\infty}}^{2\sigma-1}\left(\|Lu\|_{L^{\infty}}+t\|u_x\|_{L^{\infty}}\right)\left\|\int_{0}^{1}\int\chi(\sqrt{t}z) |z| |\partial_vw(t, (v+hz)t)|dzdh\right\|_{L_v^2}\notag \\
&\lesssim  \|u\|_{L^{\infty}}^{2\sigma-1}\left(\|Lu\|_{L^{\infty}}+t\|u_x\|_{L^{\infty}}\right)\int\chi(\sqrt{t}z) |z| dz\|\partial_vw(t, vt)\|_{L_v^2}\notag \\
&\lesssim  t^{-\frac{3}{2}}\|u\|_{L^{\infty}}^{2\sigma-1}\left(\|Lu\|_{L^{\infty}}+t\|u_x\|_{L^{\infty}}\right) \|Lu\|_{L_x^2}\notag \\
&\lesssim  t^{-\frac{3}{2}}\|u\|_{L^{\infty}}^{2\sigma-1}\|Lu\|_{L^2_x}\big(\|Lu\|_{L^{2}_x}+\|u\|_{L^{2}_x}+\|Lu_x\|_{L^{2}_x}+t\|u_x\|_{L^{\infty}}\big).\label{Remainder111R_3-L2}
	\end{align}

We now estimate $R_4$. For $\sigma\geq 1/2$, one has
\begin{align}
\big|t^{-\sigma}|\gamma|^{2\sigma}-|u(t, vt)|^{2\sigma}\big|&=\left||u(t, vt)|^{2\sigma}-|t^{-\frac{1}{2}}e^{i\phi(t, vt)}\gamma|^{2\sigma-1}\big|t^{-\frac{1}{2}}e^{i\phi(t, vt)}\gamma\big|\right|\notag \\
&\lesssim \|u\|^{2\sigma-1}_{L^{\infty}}\left||u(t, vt)|-\big|t^{-\frac{1}{2}}e^{i\phi(t, vt)}\gamma\big|\right|\nonumber.
	\end{align}
Hence, by the inequality above, \eqref{DifferenceBounds2a}, \eqref{DifferenceBounds2b} and \eqref{DifferenceBounds1b2} we get
\begin{align}
|R_4|\lesssim& t^{-\frac{3}{4}}\|v\gamma\|_{L^{\infty}}\|u\|^{2\sigma-1}_{L^{\infty}}
 \|Lu\|_{L^{2}_x} \notag\\
\lesssim& t^{-\frac{5}{4}} \left(\|Lu\|_{L^{2}_x}+\|u\|_{L^{2}_x}+\|Lu_x\|_{L^{2}_x} \right) \|u\|^{2\sigma-1}_{L^{\infty}}
 \|Lu\|_{L^{2}_x}\notag \\ 
&  +  t^{-\frac{1}{4}} \|u_x\|_{L^{\infty}} \|u\|^{2\sigma-1}_{L^{\infty}}
 \|Lu\|_{L^{2}_x} \label{Remainder111R_4},
	\end{align}
and 
\begin{align}
\|R_4\|_{L^{2}_v}\lesssim& t^{-1}\|v\gamma\|_{L^{\infty}}\|u\|^{2\sigma-1}_{L^{\infty}}
 \|Lu\|_{L^{2}_x} \notag\\
\lesssim& t^{-\frac{3}{2}} \left(\|Lu\|_{L^{2}_x}+\|u\|_{L^{2}_x}+\|Lu_x\|_{L^{2}_x} \right) \|u\|^{2\sigma-1}_{L^{\infty}}
 \|Lu\|_{L^{2}_x}\notag \\ 
&  +  t^{-\frac{1}{2}} \|u_x\|_{L^{\infty}} \|u\|^{2\sigma-1}_{L^{\infty}}
 \|Lu\|_{L^{2}_x} \label{Remainder111R_4-L2}.
	\end{align}

Then,  we obtain \eqref{AsympEq2} by collecting \eqref{Remainder111R_1}, \eqref{Remainder111R_2}, \eqref{Remainder111R_3f} and \eqref{Remainder111R_4}. And  \eqref{AsympEq2-L2}  follows directly from \eqref{Remainder111R_1-L2}, \eqref{Remainder111R_2-L2}, \eqref{Remainder111R_3-L2} and \eqref{Remainder111R_4-L2}.

\vspace{1mm}
\noindent
$\bullet$ {\bf Estimate for  $\|vR\|_{L^{\infty}}$}  
\vspace{1mm}

Denote
$$F(x)=-ix\chi(x)+2\chi'(x).$$
Observe that
\begin{align}
vR_1=&\int -vu\frac{e^{-i\phi}}{2t}\partial_x \left(-i(x-vt)\chi\left(\frac{x-vt}{\sqrt{t}}\right)+2\sqrt{t}\chi'\left(\frac{x-vt}{\sqrt{t}}\right)\right)dx\notag \\
=&-\frac{v}{2\sqrt{t}}\int u e^{-i\phi} \partial_x F\left(\frac{x-vt}{\sqrt{t}}\right)dx \notag \\
=&\frac{1}{2\sqrt{t}}\int \frac{x-vt}{t}u e^{-i\phi} \partial_x F\left(\frac{x-vt}{\sqrt{t}}\right)dx -\frac{1}{2\sqrt{t}}\int \frac{x}{t}u e^{-i\phi} \partial_x F\left(\frac{x-vt}{\sqrt{t}}\right)dx.\label{VR1TT}
	\end{align}

On one hand, as $\big|\frac{x-vt}{\sqrt{t}}\big|\lesssim 1$, we see the first term in \eqref{VR1TT} can be bounded by
\begin{align}
\left|\frac{1}{2\sqrt{t}}\int \frac{x-vt}{t}u e^{-i\phi} \partial_x F\left(\frac{x-vt}{\sqrt{t}}\right)dx\right|
\lesssim & t^{-1}\|u\|_{L^{\infty}}\| \partial_x F\|_{L^1_x}\lesssim  t^{-1}\|u\|_{L^{\infty}}.\label{VR1T1}
	\end{align}
On the other hand, for the second term in \eqref{VR1TT} we have
\begin{align}
&\frac{1}{2\sqrt{t}}\int \frac{x}{t}u e^{-i\phi} \partial_x F\left(\frac{x-vt}{\sqrt{t}}\right)dx\notag \\
=&\frac{1}{2t^{3/2}}\int (Lu-2itu_x)e^{-i\phi} \partial_x F\left(\frac{x-vt}{\sqrt{t}}\right)dx\notag \\
=&\frac{1}{2t^{3/2}}\int Lu e^{-i\phi} \partial_x F\left(\frac{x-vt}{\sqrt{t}}\right)dx\notag \\
&-\frac{i}{\sqrt{t}}\int \big(u_x(t,x)e^{-i\phi(t,x)}-u_x(t,vt)e^{-i\phi(t,vt)}\big) \partial_x F\left(\frac{x-vt}{\sqrt{t}}\right)dx\notag \\
:=&T_1+T_2.\label{VR1T2}
	\end{align}

We use H\"older inequality to estimate $T_1$.
\begin{align}
|T_1|\lesssim t^{-\frac{3}{2}}\|Lu\|_{L^2_x}\left\|t^{-\frac{1}{2}} F_x \big(t^{-\frac{1}{2}}(x-vt)\big)\right\|_{L^2_x}\lesssim t^{-\frac{7}{4}}\|Lu\|_{L^2_x}.\label{VR1T21}
	\end{align}
Utilizing the change of variables $x=vt+zt$, one can write
\begin{align}
T_2=-i\int \big(W(t,vt+zt)-W(t,vt)\big)   F_x(\sqrt{t}z)dz\nonumber
	\end{align}
where $W(t,\cdot)=u_x(t,\cdot)e^{-i\phi(t,\cdot)}$ and $\partial_vW(t, vt)=-\frac{i}{2}e^{-i\phi(t, vt)}(Lu_x)(t, vt)$. Then using the Fundamental Theorem of Calculus
\begin{align}
\big|W(t,vt+zt)-W(t,vt)\big|\lesssim |z|^{\frac{1}{2}}\|\partial_vW(t, vt)\|_{L_v^2},
\label{VRFundTheCal-u_x}
	\end{align}
we get 
\begin{align}
|T_2|\lesssim t^{-\frac{3}{4}}\|\partial_vW(t, vt)\|_{L_v^2}\lesssim t^{-\frac{5}{4}}\|Lu_x\|_{L_x^2}.\label{VR1T22}
	\end{align}
Substituting \eqref{VR1T21} and \eqref{VR1T22}  into \eqref{VR1T2}, we obtain
\begin{align}
\frac{1}{2\sqrt{t}}\left|\int \frac{x}{t}u e^{-i\phi} \partial_x F\left(\frac{x-vt}{\sqrt{t}}\right)dx \right|
\lesssim t^{-\frac{7}{4}}\|Lu\|_{L^2_x}+t^{-\frac{5}{4}}\|Lu_x\|_{L^2_x} \nonumber
	\end{align}
which together with \eqref{VR1T1} and \eqref{VR1TT} yield 
\begin{align}
\|vR_1\|_{L^{\infty}}
\lesssim t^{-1}\|u\|_{L^{\infty}_x}+t^{-\frac{7}{4}}\|Lu\|_{L^2_x}+t^{-\frac{5}{4}}\|Lu_x\|_{L^2_x}.\label{Remainder222R_1}
	\end{align}

Let us estimate $vR_2$. It is easy to see that
\begin{align}
\|vR_2\|_{L^{\infty}}
\lesssim& t^{-\frac{1}{2}} \left|\int \frac{x}{t}u |u|^{2\sigma} \overline{\Psi_v}dx\right|\notag \\
\lesssim& t^{-\frac{3}{2}} \left|\int\big(Lu-2itu_x\big)|u|^{2\sigma}\overline{\Psi_v}dx\right|\notag \\
\lesssim &t^{-\frac{3}{2}} \left|\int |u|^{2\sigma}Lu \overline{\Psi_v}dx\right|+t^{-\frac{1}{2}} \left|\int u_x  |u|^{2\sigma}\overline{\Psi_v}dx\right|\notag \\
\lesssim &t^{-\frac{3}{2}}\|u\|^{2\sigma}_{L^{\infty}} \|Lu\|_{L^{2}} \|\Psi_v\|_{L^{2}}+t^{-\frac{1}{2}}\|u_x\|_{L^{\infty}} \|u\|^{2\sigma}_{L^{\infty}} \|\Psi_v\|_{L^{1}}\notag \\
\lesssim &t^{-\frac{5}{4}}\|u\|^{2\sigma}_{L^{\infty}} \|Lu\|_{L^{2}} +\|u_x\|_{L^{\infty}} \|u\|^{2\sigma}_{L^{\infty}}.\label{Remainder222R_2}
	\end{align}

Next, we turn to estimate $vR_3$. Using the facts $|x|\approx |vt|$ and $xu=Lu-2itu_x$, we immediately get 
\begin{align}
\|vR_3\|_{L^{\infty}}
\lesssim& t^{-2}\|u\|^{2\sigma-1}_{L^{\infty}}\int|xu\overline{\Phi_{v}}|\big|xu(t,x)e^{-i\phi(t,x)}-vtu(t, vt)e^{-i\phi(t,vt)}\big|dx\notag \\
\lesssim& t^{-2} \|u\|^{2\sigma-1}_{L^{\infty}} \int |Lu-2itu_x | |\overline{\Phi_{v}}|
\big|Lu(t,x)e^{-i\phi(t,x)}-Lu(t,vt)e^{-i\phi(t,vt)}\big|dx \notag \\
&+t^{-1} \|u\|^{2\sigma-1}_{L^{\infty}} \int |Lu-2itu_x | |\overline{\Phi_{v}}|
\big|u_x(t,x)e^{-i\phi(t,x)}-u_x(t, vt)e^{-i\phi(t,vt)}\big|dx\notag \\
\lesssim &t^{-2} \|u\|^{2\sigma-1}_{L^{\infty}} \int |Lu(t, x) | |\overline{\Phi_{v}}|
\big|Lu(t,x)e^{-i\phi(t,x)}-Lu(t,vt)e^{-i\phi(t,vt)}\big|dx \notag \\
&+t^{-1} \|u\|^{2\sigma-1}_{L^{\infty}} \int |Lu(t, x)| |\overline{\Phi_{v}}|
 |u_x(t,x)|dx\notag \\
&+ \|u\|^{2\sigma-1}_{L^{\infty}} \int |u_x | |\overline{\Phi_{v}}|
\big|u_x(t,x)e^{-i\phi(t,x)}-u_x(t, vt)e^{-i\phi(t,vt)}\big|dx\notag \\
:=&VR_{3,I}+VR_{3,II}+VR_{3,III}.\label{Remainder222R_3a0}
	\end{align}

$VR_{3,I}$ is easy to be controlled by taking use of  H\"older inequality and \eqref{DiffBoundPf2b}
\begin{align}
VR_{3,I}\lesssim &t^{-2} \|u\|^{2\sigma-1}_{L^{\infty}} \big(\|Lu\|^2_{L^2_x}+\|Lu\|_{L^2_x}\|\Phi_{v}\|_{L^2_x}\|Lu\|_{L^{\infty}_x}\big) \notag \\
\lesssim &t^{-\frac{7}{4}} \|u\|^{2\sigma-1}_{L^{\infty}}\|Lu\|_{L^2_x} \big(\|Lu\|_{L^{2}_x}+\|u\|_{L^{2}_x}+\|Lu_x\|_{L^{2}_x}\big).\label{Remainder222R_3a0I}
	\end{align}
For $VR_{3,II}$, we have
\begin{align}
VR_{3,II}\lesssim &t^{-1} \|u\|^{2\sigma-1}_{L^{\infty}}  \|Lu\|_{L^2_x}\|\Phi_{v}\|_{L^2_x}\|u_x\|_{L^{\infty}_x}\notag \\
\lesssim &t^{-\frac{3}{4}}\|u\|^{2\sigma-1}_{L^{\infty}}  \|Lu\|_{L^2_x}\|u_x\|_{L^{\infty}_x}.\label{Remainder222R_3a0II}
	\end{align}
Taking variable substitution $x=vt+zt$ and using \eqref{VRFundTheCal-u_x} deduce that
\begin{align}
VR_{3,III}\lesssim & t^{\frac{1}{2}} \|u\|^{2\sigma-1}_{L^{\infty}}\|Lu_x\|_{L_x^2} \int \big|u_x(t, vt+zt)    \chi(\sqrt{t}z) 
|z|^{\frac{1}{2}}\big|dz\notag \\
 \lesssim &  t^{-\frac{1}{4}} \|u\|^{2\sigma-1}_{L^{\infty}}\|Lu_x\|_{L_x^2} \|u_x\|_{L^{\infty}_x}.\label{Remainder222R_3a0III}
	\end{align}
Combining \eqref{Remainder222R_3a0I}-\eqref{Remainder222R_3a0III} and \eqref{Remainder222R_3a0}, we get
\begin{align}
\|vR_3\|_{L^{\infty}}\lesssim& t^{-\frac{7}{4}} \|u\|^{2\sigma-1}_{L^{\infty}}\|Lu\|_{L^2_x} \big(\|Lu\|_{L^{2}_x}+\|u\|_{L^{2}_x}+\|Lu_x\|_{L^{2}_x}\big)\notag \\
&+t^{-\frac{3}{4}}\|u\|^{2\sigma-1}_{L^{\infty}}  \|Lu\|_{L^2_x}\|u_x\|_{L^{\infty}_x}+t^{-\frac{1}{4}} \|u\|^{2\sigma-1}_{L^{\infty}}\|Lu_x\|_{L_x^2} \|u_x\|_{L^{\infty}_x}.\label{Remainder222R_3}
	\end{align}

Finally, we  estimate $vR_4$. Observe that
\begin{align}
&|v|\big|t^{-\sigma}|\gamma|^{2\sigma}-|u(t, vt)|^{2\sigma}\big|\notag \\
\lesssim &\|u\|^{2\sigma-1}_{L^{\infty}_x}|v|\big|t^{-\frac{1}{2}}\gamma-u(t, vt)e^{-i\phi(t, vt)}\big|\notag \\
\lesssim &t^{-\frac{3}{2}} \|u\|^{2\sigma-1}_{L^{\infty}_x}  \int \left|\big(xue^{-i\phi}-vtu(t, vt)e^{-i\phi(t, vt)}\big)\overline{\chi\left(\frac{x-vt}{\sqrt{t}}\right)}\right|dx\notag \\
\lesssim & VR_{4, I}+VR_{4, II} \label{Remainder222R_4a0}
	\end{align}
with 
\begin{align}
VR_{4, I}:=&t^{-\frac{3}{2}} \|u\|^{2\sigma-1}_{L^{\infty}_x}\int \big(|Lu(t, x)|+|Lu(t, vt)|\big)\chi\left(\frac{x-vt}{\sqrt{t}}\right) dx\notag \\
\lesssim&\|u\|^{2\sigma-1}_{L^{\infty}_x}\big(t^{-\frac{5}{4}} \|Lu\|_{L_x^2}+t^{-1} \|Lu\|_{L^{\infty}}\big)\notag \\
\lesssim&t^{-1} \|u\|^{2\sigma-1}_{L^{\infty}_x}\big( \|Lu\|_{L_x^2}+\|u\|_{L_x^2}+\|Lu_x\|_{L_x^2}\big)
\label{Remainder222R_4a1}
	\end{align}
and 
\begin{align}
VR_{4, II}:=&t^{-\frac{1}{2}} \|u\|^{2\sigma-1}_{L^{\infty}_x}\int \big|u_xe^{-i\phi}-u_x(t, vt)e^{-i\phi(t, vt)}\big|\chi\left(\frac{x-vt}{\sqrt{t}}\right) dx\notag \\
\lesssim&\|u\|^{2\sigma-1}_{L^{\infty}_x}\|Lu_x\|_{L_x^2} \int  \chi(\sqrt{t}z) 
|z|^{\frac{1}{2}}dz\notag \\
\lesssim&t^{-\frac{3}{4}} \|u\|^{2\sigma-1}_{L^{\infty}_x}\|Lu_x\|_{L_x^2}.
\label{Remainder222R_4a2}
	\end{align}
The veracity of the last step in \eqref{Remainder222R_4a2} follows similarly from  \eqref{Remainder222R_3a0III}. Then, from \eqref{Remainder222R_4a0}-\eqref{Remainder222R_4a2} one obtains
\begin{align}
|v|\big|t^{-\sigma}|\gamma|^{2\sigma}-|u(t, vt)|^{2\sigma}\big|
\lesssim &t^{-1} \|u\|^{2\sigma-1}_{L^{\infty}_x}\big( \|Lu\|_{L_x^2}+\|u\|_{L_x^2}+\|Lu_x\|_{L_x^2}\big)\notag \\
&+t^{-\frac{3}{4}} \|u\|^{2\sigma-1}_{L^{\infty}_x}\|Lu_x\|_{L_x^2}  \nonumber
	\end{align}
which  together with \eqref{DifferenceBounds1b2} yield that
\begin{align}
\|vR_4\|_{L^{\infty}}\lesssim &\|v\gamma\|_{L^{\infty}} |v|\big|t^{-\sigma}|\gamma|^{2\sigma}-|u(t, vt)|^{2\sigma}\big|\notag \\
\lesssim &t^{-\frac{3}{2}} \|u\|^{2\sigma-1}_{L^{\infty}_x}\big( \|Lu\|_{L_x^2}+\|u\|_{L_x^2}+\|Lu_x\|_{L_x^2}\big)^2\notag \\
&+t^{-\frac{1}{2}} \|u\|^{2\sigma-1}_{L^{\infty}_x}\big( \|Lu\|_{L_x^2}+\|u\|_{L_x^2}+\|Lu_x\|_{L_x^2}\big)\|u_x\|_{L^{\infty}}\notag \\
&+t^{-\frac{5}{4}} \|u\|^{2\sigma-1}_{L^{\infty}_x}\big( \|Lu\|_{L_x^2}+\|u\|_{L_x^2}+\|Lu_x\|_{L_x^2}\big)\|Lu_x\|_{L^{2}_x}\notag \\
&+t^{-\frac{1}{4}} \|u\|^{2\sigma-1}_{L^{\infty}_x}\|Lu_x\|_{L_x^2}\|u_x\|_{L_x^{\infty}}. \label{Remainder222R_4}
	\end{align}

Collecting  \eqref{Remainder222R_1}, \eqref{Remainder222R_2}, \eqref{Remainder222R_3} and  \eqref{Remainder222R_4}, we deduce \eqref{AsympEq3}. So we finish the proof of this lemma.
\end{proof}

\subsection{Energy bounds on  $Lu$ and $Lu_x$}

 In this subsection, we show estimates for the growth of  $\|Lu\|_{L^{2}_x}$ and $\|Lu_x\|_{L^{2}_x}$ under the assumptions
\begin{align}
\|u\|_{L^{\infty}_x}\leq D\epsilon \langle t \rangle^{-\frac{1}{2}}, \hspace{10mm} \|u_x\|_{L^{\infty}_x}\leq \sqrt{D\epsilon} \langle t \rangle^{-\frac{1}{2}}, \label{BootstrapAssump}
	\end{align}
where $D>0$ is a  universal constant.

 Acting operator $L$ on both sides of \eqref{gDNLS} and using the identities in \eqref{VecField} deduce the equation for $Lu$
\begin{align}
i\partial_t Lu +\partial^2_{x}Lu+i\partial_x\big((\sigma+1)|u|^{2\sigma}Lu-\sigma |u|^{2\sigma-2}u^{2}\overline{Lu}\big)-i|u|^{2\sigma}u=0. \label{gDNLS-Lu}
	\end{align}
Then, applying the energy method in the usual way, we get
\begin{align}
\frac{d}{dt}\|Lu\|^2_{L^{2}_x}&=(\sigma+1)\int |u|^{2\sigma}\partial_{x}(|Lu|^2)dx+\sigma\text{Re}\int |u|^{2\sigma-2}u^{2}\partial_{x}(\overline{Lu}^2)dx \notag \\
&\hspace{10mm}+2\text{Re}\int |u|^{2\sigma}u\overline{Lu}dx \notag \\
&=-(\sigma+1)\int\partial_{x}(|u|^{2\sigma}) |Lu|^2 dx-\sigma\text{Re}\int \partial_{x}(|u|^{2\sigma-2}u^{2})\overline{Lu}^2dx \notag \\
&\hspace{10mm}+2\text{Re}\int |u|^{2\sigma}u\overline{Lu}dx \nonumber
	\end{align}
which by using H\"older inequality and \eqref{BootstrapAssump} yields
\begin{align}
 \frac{d}{dt}\|Lu\|^2_{L^{2}_x} &\leq(2\sigma+1)^2 \|u\|^{2\sigma-1}_{L^{\infty}_x}\|u_x\|_{L^{\infty}_x}\|Lu\|^2_{L^{2}_x}+2\|u\|^{2\sigma}_{L^{\infty}_x}\|u\|_{L^{2}_x}\|Lu\|_{L^{2}_x} \notag \\
&\leq(2\sigma+1)^2 (D\epsilon)^{2\sigma-\frac{1}{2}} \langle t \rangle^{-\sigma}\|Lu\|^2_{L^{2}_x}+2\big(D\epsilon \langle t \rangle^{-\frac{1}{2}}\big)^{2\sigma}\epsilon \|Lu\|_{L^{2}_x} \label{gDNLS-Lu1}
	\end{align}
provided that $\sigma\geq \frac{1}{2}$.

Set $y(t)=\|Lu(t)\|_{L^{2}_x}$. From \eqref{gDNLS-Lu1}, we have
\begin{align}
y'+p(t)y\leq B(t) \label{gDNLS-Lu2}
	\end{align}
with $p(t)=-(D\epsilon)^{2\sigma-\frac{1}{2}} \langle t \rangle^{-\sigma}$, $ B(t)=2\big(D\epsilon \langle t \rangle^{-\frac{1}{2}}\big)^{2\sigma}\epsilon $ and $y(0)=\|xu_0\|_{L^{2}_x}\leq \epsilon$. It is easy to see from \eqref{gDNLS-Lu2} that
$$y(t)\leq e^{-\int_0^tp(s)ds}y(0)+e^{-\int_0^tp(s)ds}\int_0^te^{\int_0^sp(\tau)d\tau}B(s)ds.$$
Hence,
\begin{equation}
	\|Lu(t)\|_{L^{2}_x}\lesssim\left\{
	\begin{aligned}
		&\epsilon  \langle t \rangle^{5(D\epsilon)^{\frac{3}{2}}},  \hspace{5.78mm} \sigma=1,\\
		&\epsilon,  \hspace{20mm} \sigma>1. \label{gDNLS-Lu3} \\
	\end{aligned}
	\right.
\end{equation}
By similar argument, we also have
\begin{equation}
	\|Lu_x(t)\|_{L^{2}_x}\lesssim\left\{
	\begin{aligned}
		&\epsilon   \langle t \rangle^{5(D\epsilon)^{\frac{3}{2}}},  \hspace{5.78mm} \sigma=1,\\
		&\epsilon,  \hspace{20mm} \sigma>1. \label{gDNLS-Lu4} \\
	\end{aligned}
	\right.
\end{equation}
In short, \eqref{gDNLS-Lu3} and \eqref{gDNLS-Lu4} tell us that $\|Lu(t)\|_{H^{1}_x}\lesssim \epsilon  \langle t \rangle^{5(D\epsilon)^{\frac{3}{2}}}$ if $\sigma=1$, and $\|Lu(t)\|_{H^{1}_x}\lesssim \epsilon $ if $\sigma>1$.

\subsection{Proof of Theorem \ref{MainResult2}}

We use the bootstrap argument to complete the proof of Theorem \ref{MainResult2} in this subsection. At first, we make two bootstrap assumptions on $u$ and its derivative $u_x$, see \eqref{BootstrapAssump}:
\begin{align}
\|u\|_{L^{\infty}_x}\leq D\epsilon  \langle t \rangle^{-\frac{1}{2}}, \hspace{10mm} \|u_x\|_{L^{\infty}_x}\leq \sqrt{D\epsilon} \langle t \rangle^{-\frac{1}{2}},\nonumber
	\end{align}
where $D$ is a sufficiently large universal constant to be chosen later. Then our task will be to improve the bootstrap bound  on an arbitrarily large time interval $[0, T]$ where the solution $u$ exists.  Once this is achieved, a standard
continuity argument shows that the solution $u$ is global in time and satisfies these bounds.  The asymptotic behavior of the solution can be easily got from the approximate ODE and error estimates, thereby concluding the proof of the theorem.

Let us first consider the case $\sigma=1$.

\vspace{1mm}
\noindent
$\bullet$ {\bf Dispersive decay}  
\vspace{1mm}

 We know from Lemma  \ref{DifferenceBounds} that  $\gamma$ is a good approximation of $u$. In particular, by \eqref{DifferenceBounds2a} and \eqref{gDNLS-Lu3} we have
\begin{align}
\|u(t)\|_{L^{\infty}}\leq & \|u(t)-t^{-\frac{1}{2}}e^{i\phi(t,vt)}\gamma(t, v)\|_{L^{\infty}}+t^{-\frac{1}{2}}\|\gamma(t, v)\|_{L^{\infty}} \notag \\
\lesssim  & t^{-\frac{3}{4}}\|Lu\|_{L^{2}_x}+t^{-\frac{1}{2}}\|\gamma(t, v)\|_{L^{\infty}}\notag \\
\lesssim  &\epsilon \langle t \rangle^{5(D\epsilon)^{\frac{3}{2}}-\frac{3}{4}}
+t^{-\frac{1}{2}}\|\gamma(t, v)\|_{L^{\infty}}. \label{FinalPf1}
	\end{align}

In order to estimate $\|\gamma(t, v)\|_{L^{\infty}}$, we use the asymptotic equation  \eqref{AsympEq1}. By using the energy method, one gets
	\begin{align}
\int_1^t\frac{d}{ds}|\gamma(s,v)|^2ds=2\text{Re}\int_1^tiR(s,v)\bar{\gamma}(s,v)ds  \hspace{5mm}\text{for} \hspace{3mm} t>1,\nonumber
	\end{align}
which implies that
	\begin{align}
|\gamma(t,v)| \lesssim |\gamma(1,v)|+ \int_1^t|R(s,v)|ds.\label{FinalPf2}
	\end{align}

On the one hand, it follows easily from \eqref{DifferenceBounds1a} and \eqref{BootstrapAssump} that
	\begin{align}
\|\gamma(1)\|_{L^{\infty}}\lesssim    \|Lu(1)\|_{L^{2}_x}+\|u(1)\|_{L^{2}_x}\lesssim  \epsilon.\label{FinalPf2a}
	\end{align}
On the other hand, according to \eqref{AsympEq2}, \eqref{BootstrapAssump}, \eqref{gDNLS-Lu3} and \eqref{gDNLS-Lu4} we get
	\begin{align}
\|R\|_{L^{\infty}}\lesssim  \epsilon \langle t \rangle^{-\frac{5}{4}+5(D \epsilon)^{\frac{3}{2}}}+ D^3\epsilon^3 \langle t \rangle^{-\frac{3}{2}}+  D^{\frac{3}{2}}\epsilon^{\frac{5}{2}} \langle t \rangle^{-\frac{5}{4}+5(D \epsilon)^{\frac{3}{2}}}+ D\epsilon^3 \langle t \rangle^{-\frac{7}{4}+10(D \epsilon)^{\frac{3}{2}}}.\label{FinalPf2b}
	\end{align}
Hence,
	\begin{align}
\int_1^t|R(s,v)|ds \lesssim   \epsilon+D^3\epsilon^3+ D^{\frac{3}{2}}\epsilon^{\frac{5}{2}} +D\epsilon^3,\label{FinalPf2c}
	\end{align}
as long as $5(D \epsilon)^{\frac{3}{2}}<\frac{1}{4}$ so that each term  in \eqref{FinalPf2b} is integrable with respect to $t$. Choosing $D=\frac{1}{10\epsilon}$, then from \eqref{FinalPf2}-\eqref{FinalPf2c} we see that
	\begin{align}
 |\gamma(t,v)| \lesssim \epsilon+D^3\epsilon^3+ D^{\frac{3}{2}}\epsilon^{\frac{5}{2}} +D\epsilon^3.\label{FinalPf2f}
	\end{align}
Combining \eqref{FinalPf2f} with \eqref{FinalPf1}, we obtain
\begin{align}
\|u(t)\|_{L^{\infty}}
\lesssim  \big(1+ D/100\big) \epsilon  \langle t \rangle^{-\frac{1}{2}}<D\epsilon  \langle t \rangle^{-\frac{1}{2}}\nonumber
	\end{align}
by taking $\epsilon$  small enough. This closes the first part of the bootstrap argument.

Arguing similarly, we get by using \eqref{DifferenceBounds2c} that
\begin{align}
\|u_x(t)\|_{L^{\infty}}\leq & \|u_x(t)-\frac{i}{2}t^{-\frac{1}{2}}e^{i\phi(t,vt)}v\gamma(t, v)\|_{L^{\infty}}+t^{-\frac{1}{2}}\|v\gamma(t, v)\|_{L^{\infty}} \notag \\
\lesssim  & t^{-\frac{3}{4}}\left(\|Lu\|_{L^{2}_x}+\|Lu_x\|_{L^{2}_x}\right)+t^{-\frac{1}{2}}\|v\gamma(t, v)\|_{L^{\infty}}\notag \\
\lesssim  &\epsilon  \langle t \rangle^{5(D\epsilon)^{\frac{3}{2}}-\frac{3}{4}}+t^{-\frac{1}{2}}\|v\gamma(t, v)\|_{L^{\infty}}. \label{FinalPfXX1}
	\end{align}
From \eqref{FinalPf2}, one has
	\begin{align}
 |v\gamma(t,v)| \lesssim |v\gamma(1,v)|+ \int_1^t|vR(s,v)|ds.\label{FinalPfXX2}
	\end{align}
And one easily sees from  \eqref{AsympEq3} that
\begin{align}
\|vR\|_{L^{\infty}} \lesssim& D \epsilon  \langle t \rangle^{-\frac{3}{2}}+
  \epsilon \langle t \rangle^{-\frac{5}{4}+5(D\epsilon)^{\frac{3}{2}}}+ D^2\epsilon^3 \langle t \rangle^{-\frac{9}{4}+5(D\epsilon)^{\frac{3}{2}}}+(D\epsilon)^{\frac{5}{2}} \langle t \rangle^{-\frac{3}{2}} \notag \\
&+  D\epsilon^3 \langle t \rangle^{-\frac{7}{4}+10(D\epsilon)^{\frac{3}{2}}}+ D^{\frac{3}{2}}\epsilon^{\frac{5}{2}} \langle t \rangle^{-\frac{5}{4}+5(D\epsilon)^{\frac{3}{2}}}.\label{FinalPfXX2a}
	\end{align}
 Collecting \eqref{FinalPfXX1}-\eqref{FinalPfXX2a} yields 
\begin{align}
\|u_x(t)\|_{L^{\infty}}
\lesssim  \big(1+ D/27\big) \epsilon  \langle t \rangle^{-\frac{1}{2}}< \sqrt{D\epsilon}  \langle t \rangle^{-\frac{1}{2}}.\nonumber
	\end{align}
Therefore, we close both bootstrap assumptions.

\vspace{1mm}
\noindent
$\bullet$ {\bf Asymptotic behavior}  
\vspace{1mm}

For each $v$,  there exists a function $W(v)$ such that the solution of 
\begin{align}
i\widetilde{\gamma}_t=\frac{vt^{-1}}{2}|\widetilde{\gamma}|^2\widetilde{\gamma} \label{unperturbedODE}
	\end{align}
can be expressed as
\begin{align}
\widetilde{\gamma}(t,v)=W(v)e^{-i\frac{v}{2}|W(v)|^2\log t}.\nonumber
	\end{align}
Combining   \eqref{AsympEq2} and  \eqref{AsympEq2-L2} with  \eqref{MainResult2b}, \eqref{gDNLS-Lu3}  and \eqref{gDNLS-Lu4}, we see that
\begin{align}
\|R(t,v)\|_{L^{\infty}}
\lesssim   \epsilon  \langle t \rangle^{-\frac{5}{4}+5(D\epsilon )^{\frac{3}{2}}}, \hspace{9mm} \|R(t,v)\|_{L^{2}_v}
\lesssim   \epsilon  \langle t \rangle^{-\frac{3}{2}+5(D\epsilon )^{\frac{3}{2}}}.\label{remainder-Fi}
	\end{align}
For each $v$, since $R(t,v)$ in uniformly integrable in time, $\gamma(t,v)$ is well approximated at infinity by a solution to the unperturbed ODE \eqref{unperturbedODE}, in the sense that
\begin{align}
\gamma(t,v)=W(v)e^{-i\frac{v}{2}|W(v)|^2\log t}+O_{L^{\infty}_v}\left(\epsilon  \langle t \rangle^{-\frac{1}{4}+5(D\epsilon )^{\frac{3}{2}}}\right).\label{AsymBehav1}
	\end{align}
Integrating the $L_v^2$ part of \eqref{remainder-Fi} leads to a similar $L_v^2$ bound
\begin{align}
\gamma(t,v)=W(v)e^{-i\frac{v}{2}|W(v)|^2\log t}+O_{L^{2}_v}\left(\epsilon  \langle t \rangle^{-\frac{1}{2}+5(D\epsilon )^{\frac{3}{2}}}\right).\label{AsymBehav2}
	\end{align}
So, it is easy to see that \eqref{MainResult2c} follows directly from   \eqref{DifferenceBounds2a}, \eqref{DifferenceBounds2b}, \eqref{AsymBehav1} and \eqref{AsymBehav2}. Similarly,  \eqref{MainResult2d} follows  from   \eqref{DifferenceBounds3a}, \eqref{DifferenceBounds3b}, \eqref{AsymBehav1} and \eqref{AsymBehav2}.  

It remains to establish the regularity of $W$. From \eqref{AsymBehav2}, \eqref{DifferenceBounds1aa} and the $L^2$-conservation law, we have 
$$\|W\|_{L_v^2}\lesssim \|\gamma(t,v)\|_{L_v^2}\lesssim \|u\|_{L_x^2}\lesssim \epsilon.$$

On the other hand, from  \eqref{AsympEq3}, \eqref{MainResult2b}, \eqref{gDNLS-Lu3}  and \eqref{gDNLS-Lu4},  we get
\begin{align}
v\gamma(t,v)=vW(v)e^{-i\frac{v}{2}|W(v)|^2\log t}+O_{L^{\infty}_v}\left(\epsilon  \langle t \rangle^{-\frac{1}{4}+5(D\epsilon )^{\frac{3}{2}}}\right) \nonumber
	\end{align}
which together with \eqref{DifferenceBounds1b2} implies that
\begin{align}
\|vW(v)\|_{L_v^{\infty}}\lesssim \|v\gamma(t,v)\|_{L_v^{\infty}}+\epsilon  \langle t \rangle^{-\frac{1}{4}+5(D\epsilon )^{\frac{3}{2}}}\lesssim \epsilon.\label{AsymBehav1v}
	\end{align}
Then, from \eqref{AsymBehav1}-\eqref{AsymBehav1v} and \eqref{DifferenceBounds1aa}, one has
\begin{align}
&\left\|W(v)-\gamma(t,v)e^{i\frac{v}{2}|\gamma(t,v)|^2\log t}\right\|_{L_v^2}\notag \\
\lesssim& \left\|\gamma(t,v)e^{-i\frac{v}{2}|W(v)|^2\log t}-\gamma(t,v)e^{i\frac{v}{2}|\gamma(t,v)|^2\log t}\right\|_{L_v^2} +\epsilon  \langle t \rangle^{-\frac{1}{2}+5(D\epsilon )^{\frac{3}{2}}}\notag \\
\lesssim& \|\gamma(t,v)\|_{L_v^2} \left\|v|\gamma(t,v)|^2-v|W(v)|^2\right\|_{L_v^{\infty}}\log t+\epsilon  \langle t \rangle^{-\frac{1}{2}+5(D\epsilon )^{\frac{3}{2}}} \notag \\
\lesssim& \epsilon  \langle t \rangle^{-\frac{1}{4}+5(D\epsilon )^{\frac{3}{2}}}\log t.
\label{AsymBehav-W1}
	\end{align}
While by \eqref{DifferenceBounds1a}, \eqref{DifferenceBounds1aa} and \eqref{DifferenceBounds1b2} we get
\begin{align}
&\left\|\partial_v \big(\gamma(t,v)e^{i\frac{v}{2}|\gamma(t,v)|^2\log t} \big)\right\|_{L_v^2}\notag \\
\lesssim& \|\partial_v \gamma(t,v)\|_{L_v^2}+ \left\|\gamma(t,v) |\gamma(t,v)|^2\right\|_{L_v^2}\log t + \left\|v|\gamma(t,v)|^2\partial_v\gamma(t,v) \right\|_{L_v^2}\log t  \notag \\
\lesssim& \epsilon  \langle t \rangle^{5(D\epsilon )^{\frac{3}{2}}}\log t.
\label{AsymBehav-W2}
	\end{align}
Hence, according to \eqref{AsymBehav-W1} and \eqref{AsymBehav-W2}, we have
$$W(v)=O_{H_v^1}(\epsilon  \langle t \rangle^{5(D\epsilon )^{\frac{3}{2}}}\log t)+O_{L_v^2}(\epsilon  \langle t \rangle^{-\frac{1}{4}+5(D\epsilon )^{\frac{3}{2}}}\log t).$$
By interpolation we see that for large enough $C$  
$$\|W(v)\|_{H_v^{1-C\epsilon ^{\frac{3}{2}}}}\lesssim \epsilon. $$

If $\sigma>1$, by \eqref{AsympEq2}, \eqref{AsympEq3} and the fact $\|Lu(t)\|_{H^{1}_x}\lesssim \epsilon $,  we  know that the decay of $\|R\|_{L^{\infty}}$ and  $\|vR\|_{L^{\infty}}$ are  $\langle t \rangle^{-\frac{5}{4}}$. Hence, all the terms are integrable.  The bootstrap argument is obviously feasible. 

Moreover,  there exists a function $\widetilde{W}(v)\in H^{1}_v$ such that
$$\widetilde{\gamma}(t,v)=\widetilde{W}(v)e^{-i\frac{v}{2}|\widetilde{W}(v)|^{2\sigma}t^{1-\sigma}}$$
is the solution of 
$$i\widetilde{\gamma}_t=\frac{vt^{-\sigma}}{2}|\widetilde{\gamma}|^{2\sigma}\widetilde{\gamma}.$$
In this case,
\begin{align}
\gamma(t,v)=\widetilde{W}(v)e^{-i\frac{v}{2}|\widetilde{W}(v)|^{2\sigma}t^{1-\sigma}}+err_{\sigma}\label{AsymBehav3}
	\end{align}
with
$$err_{\sigma}\in O_{L^{\infty}_v}\left(\epsilon  \langle t \rangle^{-\frac{1}{4}}\right)\cap
O_{L^{2}_v}\left(\epsilon  \langle t \rangle^{-\frac{1}{2}}\right).$$
Then the asymptotic expansions in \eqref{MainResult2e} and \eqref{MainResult2f} can be derived from  \eqref{DifferenceBounds2a}-\eqref{DifferenceBounds3a}  and \eqref{AsymBehav3}.

  This completes the proof of the theorem.


\end{document}